\documentclass[a4paper,11pt]{amsart}

\usepackage{amsmath} 
\usepackage{amssymb} 
\usepackage{amscd} 
\usepackage{amsxtra}

\newcommand{\C}{\mathbb C}
\newcommand{\R}{\mathbb R}
\newcommand{\Z}{\mathbb Z}
\newcommand{\Q}{\mathbb Q}

\newcommand{\F}{\mathbb F}
\newcommand{\Flim}{\F_{\rm lim}} 
\newcommand{\tbF}{\widetilde{\mathbb{F}}}

\newcommand{\Proj}{\mathbb P}

\newcommand{\D}{\mathbb D} 
 
\newcommand{\HH}{\mathbb H} 
 
\newcommand{\J}{\mathbb{J}}

\newcommand{\seminf}{$\frac{\infty}{2}$} 

\newcommand{\ev}{\operatorname{ev}}
\newcommand{\age}{\operatorname{age}}
\newcommand{\Hom}{\operatorname{Hom}}
\newcommand{\End}{\operatorname{End}}

\newcommand{\Pic}{\operatorname{Pic}}
\newcommand{\Ker}{\operatorname{Ker}}

\newcommand{\Image}{\operatorname{Im}}

\newcommand{\id}{\operatorname{id}}
\newcommand{\Grading}{\operatorname{\mathsf{Gr}}}

\newcommand{\unit}{\operatorname{\boldsymbol{1}}}

\newcommand{\Eff}{\operatorname{Eff}}

\newcommand{\ch}{\operatorname{ch}} 
\newcommand{\tch}{\widetilde{\operatorname{ch}}} 
\newcommand{\Td}{\operatorname{Td}} 
\newcommand{\tTd}{\widetilde{\operatorname{Td}}}

\newcommand{\pr}{\operatorname{pr}}

\newcommand{\Aut}{\operatorname{Aut}} 
\newcommand{\Gr}{\operatorname{Gr}} 
\newcommand{\inv}{\operatorname{inv}}
 
\newcommand{\pt}{\operatorname{pt}} 
 
\newcommand{\Lie}{\operatorname{Lie}} 
\newcommand{\Euler}{\operatorname{Euler}}

\newcommand{\sfT}{\mathsf{T}} 
\newcommand{\sfH}{\mathsf{H}} 

\newcommand{\cF}{\mathcal{F}}

\newcommand{\cA}{\mathcal{A}}
\newcommand{\cU}{\mathcal{U}}
\newcommand{\cO}{\mathcal{O}}

\newcommand{\cX}{\mathcal{X}}
\newcommand{\cH}{\mathcal{H}}
\newcommand{\tcH}{\widetilde{\mathcal{H}}}

\newcommand{\cM}{\mathcal{M}}

\newcommand{\cR}{\mathcal{R}}

\newcommand{\cV}{\mathcal{V}}
\newcommand{\cQ}{\mathcal{Q}}  

\newcommand{\cC}{\mathcal{C}}

\newcommand{\tcF}{\widetilde{\mathcal{F}}}

\newcommand{\cRz}{\mathcal{R}^{(0)}}

\newcommand{\hatnabla}{\widehat{\nabla}}

\newcommand{\hK}{\widehat{K}}
\newcommand{\tC}{\widetilde{C}}
 
\newcommand{\hGamma}{\widehat{\Gamma}}

\newcommand{\fra}{\mathfrak{a}}

\newcommand{\gl}{\mathfrak{gl}}

\newcommand{\ov}{\overline}
\newcommand{\oi}{{\overline{\imath}}}
\newcommand{\oj}{{\overline{\jmath}}}

\newcommand{\iu}{\pmb{\mathtt{i}}}

\newtheorem{theorem}{Theorem}[section]
\newtheorem{lemma}[theorem]{Lemma}
\newtheorem{proposition}[theorem]{Proposition}
\newtheorem{proposition-conjecture}[theorem]{Proposition-Conjecture}
\newtheorem{definition-proposition}[theorem]{Definition-Proposition}

\theoremstyle{definition}
\newtheorem{definition}[theorem]{Definition}
\newtheorem{remark}[theorem]{Remark}

\newtheorem{assumption}[theorem]{Assumption}

\def\pair#1#2{\langle #1,#2\rangle}
\def\parfrac#1#2{\frac{\partial{#1}}{\partial #2}}
\def\corr#1{\left\langle #1 \right\rangle}

\begin{document}

\title{$tt^*$-geometry in quantum cohomology} 
\author{Hiroshi Iritani}
\address{Faculty of Mathematics, Kyushu University, 6-10-1, 
Hakozaki, Higashiku, Fukuoka, 812-8581, Japan.}
\email{iritani@math.kyushu-u.ac.jp}
\address{Department of Mathematics, Imperial College London, 
Huxley Building, 180, Queen's Gate, London, 
SW7 2AZ, United Kingdom.}
\subjclass[2000]{Primary~14N35, 53D45; Secondary~32G20, 32G34, 34M55}
\email{h.iritani@imperial.ac.uk}  
\begin{abstract} 
We study possible real structures in the space 
of solutions to the quantum differential equation. 
We show that, under mild conditions, 
a real structure in orbifold quantum cohomology 
yields a pure and polarized $tt^*$-geometry  
near the large radius limit. 
We compute an example of $\Proj^1$ 
which is pure and polarized 
over the whole K\"{a}hler moduli space 
$H^2(\Proj^1,\C^*)$.  
\end{abstract} 
\maketitle 

\section{Introduction} 
The quantum cohomology is a family $(H^*(X), \circ_\tau)$ 
of commutative rings parametrized by $\tau \in H^*(X)$ 
and satisfies the following ``integrability":  
A family $\nabla^z$ of connections (Dubrovin connection) 
on the trivial vector bundle 
$H^*(X) \times H^*(X) \to H^*(X)$ 
\[
\nabla^z = d + \frac{1}{z} \sum_{i=1}^N 
(\phi_i \circ_\tau) d t^i,   \quad 
z\in \C^* 
\] 
is flat for all $z\in \C^*$. 
Here we take a basis $\{\phi_i\}_{i=1}^N$ of $H^*(X)$ 
and linear co-ordinates  $\{t^i\}_{i=1}^N$ dual to it 
and $\tau = \sum_{i=1}^N t^i \phi_i$. 
The Dubrovin connection $\nabla^z$ is extended 
to a flat connection $\hatnabla$ 
over $\{(\tau,z) \in H^*(X)\times \C^*\}$ 
and this defines a local system $R$ of 
$\C$-vector spaces over $H^*(X)\times \C^*$.  
In this paper, we consider its real structure --- 
a sub local system $R_\R\subset R$ of $\R$-vector spaces. 
When written in a frame compatible with 
a real structure, the holomorphic 
connection $\nabla^z$ gains 
antiholomorphic part 
and gives rise to $tt^*$-geometry 
or topological--anti-topological fusion 
\cite{cecotti-vafa-top-antitop, 
dubrovin-fusion, hertling-tt*}. 

The study of real structures in quantum cohomology 
is motivated from mirror symmetry. 
The quantum cohomology of a Calabi-Yau threefold 
$X$ defines a variation of Hodge structure (VHS) 
over $H^{1,1}(X)$ \cite{morrison-mathaspects, cox-katz}
\[
F^3 \subset F^2 \subset F^1 \subset F^0 = H^{*,*}(X), 
\quad F^p = \bigoplus_{k\ge p} H^{3-k,3-k}(X)   
\]
by the Dubrovin connection $\nabla^z$. 
Mirror symmetry conjecture says that this is 
isomorphic to the VHS 
$\check{F}^p = \bigoplus_{k\ge p} H^{k,3-k}(Y)$ 
of a mirror Calabi-Yau $Y$ 
over the complex moduli space of $Y$. 
While the VHS of $Y$ has a natural real structure 
$H^3(Y,\R)$ (and an integral structure $H^3(Y,\Z)$), the VHS 
associated to the quantum cohomology of $X$ 
does not seem to have a natural real structure. 
In the companion paper \cite{iritani-Int}, 
we studied mirror symmetry for toric orbifolds.  
The calculation there suggested that the $K$-theory and 
the $\hGamma$-class of $X$ define a natural integral 
(hence real) structure on the quantum cohomology VHS. 
The same rational structure was also proposed 
by Katzarkov-Kontsevich-Pantev \cite{KKP} independently.

In this paper, we use the language of 
\emph{semi-infinite variation of Hodge structure} 
(henceforth \seminf VHS) 
due to Barannikov \cite{barannikov-qpI,barannikov-proj} 
to deal also with non Calabi-Yau case. 
Here we briefly explain the \seminf VHS of 
quantum cohomology. 
Let $L(\tau,z)$ be the fundamental solution 
for Dubrovin-flat sections $\nabla^z s =0$: 
\[
L \colon H^*(X)\times \C^* \to \End(H^*(X)), \quad 
\nabla^z L(\tau,z) \phi =0, \quad \phi \in H^*(X), 
\]
which is explicitly given by the gravitational descendants 
(see (\ref{eq:fundamentalsol_L})).   
Following Coates-Givental \cite{coates-givental}, 
we introduce an infinite dimensional vector space $\cH^X$ by 
\[
\cH^X := H^*(X) \otimes \cO(\C^*) 
\]
where $\cO(\C^*)$ denotes the space of 
holomorphic functions on $\C^*$ with co-ordinate $z$. 
We identify $\cH^X$ with the space of $\nabla^z$-flat sections 
by the map $\cH^X \ni \alpha(z) \mapsto L(\tau,z) \alpha(z)$.  
We define the family $\F_\tau$ 
of ``semi-infinite" subspaces of $\cH^X$ as  
\[
\F_\tau := L(\tau,z)^{-1}(H^*(X)\otimes \cO(\C)) \subset \cH^X, \quad 
\tau \in H^*(X), 
\]
where $\cO(\C)$ denotes the space of holomorphic 
functions on $\C$. 
The semi-infinite flag $\cdots\subset 
z^{-1} \F_\tau \subset \F_\tau \subset z \F_\tau \subset \cdots$ 
satisfies properties 
analogous to the usual finite dimensional VHS: 
\begin{alignat}{2}
\label{eq:introd_Griffiths}
&\parfrac{}{t^i} \F_\tau \subset z^{-1} \F_\tau 
\quad  
&& \text{(Griffiths Transversality)}   \\
\label{eq:introd_bilinearrel} 
&(\F_\tau,\F_\tau)_{\cH^X} \subset \cO(\C) 
\quad 
&& \text{(Bilinear Relations)} 
\end{alignat} 
where the pairing $(\cdot,\cdot)_{\cH^X}$ 
is defined by $(\alpha,\beta)_{\cH^X} = 
\int_X \alpha(-z) \cup \beta(z)$ for $\alpha,\beta\in \cH^X$.  
We call this (the moving subspace realization of) 
the \emph{quantum cohomology \seminf VHS}.

A real structure of the quantum cohomology 
(a sub $\R$-local system $R_\R$ of $\hatnabla$) 
induces a subspace $\cH^X_\R$ of $\cH^X$ 
\[
\cH^X_\R :=\left \{\alpha(z) \in \cH^X\; ; \; 
L(\tau,z)\alpha(z) \in R_{\R,(\tau,z)} \text{ when } 
|z|=1\right\}
\]
and the involution $\kappa_\cH \colon \cH^X  \to \cH^X$ 
fixing $\cH^X_\R$ and satisfying   
$\kappa_\cH(f(z) \alpha ) = \ov{f(1/\ov{z})} 
\kappa_\cH(\alpha)$ for $f(z) \in \cO(\C^*)$. 
For a ``good" real structure, 
we hope the following properties:  
\begin{alignat}{2}
\label{eq:introd_hodgedecomp} 
&\F_\tau \oplus z^{-1} \kappa_\cH(\F_\tau) = \cH^X, 
\quad
&&\text{(Hodge Decomposition)} \\
\label{eq:introd_bilinearineq}  
&(\kappa_\cH(\alpha),\alpha)_{\cH^X}>0, \quad 
\alpha\in \F_\tau \cap \kappa_{\cH}(\F_\tau)
\setminus\{0\}.   
\quad 
&&\text{(Bilinear Inequality)}  
\end{alignat} 
When $X$ is Calabi-Yau, these 
(\ref{eq:introd_Griffiths}), (\ref{eq:introd_bilinearrel}), 
(\ref{eq:introd_hodgedecomp}), (\ref{eq:introd_bilinearineq}) 
are translations of the corresponding properties 
for the finite dimensional VHS. 
The properties (\ref{eq:introd_hodgedecomp}) and 
(\ref{eq:introd_bilinearineq}) are called 
\emph{pure} and \emph{polarized} respectively.
Our main theorem states that 
(\ref{eq:introd_hodgedecomp}), 
(\ref{eq:introd_bilinearineq}) indeed hold  
near the ``large radius limit" 
\emph{i.e.} $\tau= -x \omega$, $\Re(x)\to\infty$ 
for some K\"{a}hler class $\omega$, 
under mild assumptions on the real 
structures: 

\begin{theorem}[Theorem \ref{thm:pure_polarized}] 
\label{thm:introd_pure_polarized} 
Assume that a real structure is invariant 
under the monodromy (Galois) transformations given by  
$G^{\cH}(\xi), \xi\in H^2(X,\Z)$ 
(see (\ref{eq:Galois_H}) and 
Proposition \ref{prop:char_A_real_int_str}). 
If the condition (\ref{eq:kappa_induces_Inv}) 
(which is empty when $X$ is a manifold) holds,  
$\F_\tau$ is pure (\ref{eq:introd_hodgedecomp}) 
near the large radius limit. 
If moreover the condition 
(\ref{eq:leadingterm_kappaV}) holds and 
$H^*(X) = \bigoplus_p H^{p,p}(X)$, 
$\F_\tau$ is polarized (\ref{eq:introd_bilinearineq}) 
near the large radius limit.  
\end{theorem} 

In the theorem above, 
we allow $X$ to be an orbifold or 
a smooth Deligne-Mumford stack.   
See Theorem \ref{thm:pure_polarized} 
for a more precise statement. 
Given a real structure 
satisfying (\ref{eq:introd_hodgedecomp}) 
and (\ref{eq:introd_bilinearineq}),  
quantum cohomology gives a Hermitian vector 
bundle with a connection $D$ 
and endomorphisms $\kappa, C,\tC, \cU, \cQ$ 
satisfying the \emph{$tt^*$-equations} (Proposition \ref{prop:CV-str}). 
This structure (\emph{$tt^*$-geometry}) was discovered by 
Cecotti-Vafa \cite{cecotti-vafa-top-antitop, 
cecotti-vafa-classification}  
and has been studied by Dubrovin \cite{dubrovin-fusion} 
and Hertling \cite{hertling-tt*}.  
$tt^*$-geometry also gives an example of 
a harmonic bundle or a twistor structure 
of Simpson \cite{simpson-mixedtwistor}.  
Closely related results 
have been shown in a more abstract setting for 
TERP structures in \cite{hertling-tt*, hertling-sevenheck}.  
In fact, when $X$ is Fano and 
the K\"{a}hler class $\omega$ is $c_1(X)$, 
the conclusions of Theorem \ref{thm:introd_pure_polarized} 
can be deduced from \cite[Theorem 7.3]{hertling-sevenheck}.

This paper is structured as follows. 
In Section \ref{sec:generalVHS}, 
we introduce real structures for 
a general graded \seminf VHS. 
This section is a translation of the work of 
Hertling \cite{hertling-tt*} 
in terms of Barannikov's semi-infinite Hodge structure.  
In Section \ref{sec:A-model}, 
we study real structures in (orbifold) quantum cohomology and 
prove Theorem \ref{thm:introd_pure_polarized}. 
We also give a review of the $\hGamma$-real structure 
given by $K$-theory 
\cite{iritani-realint-preprint, iritani-Int, KKP} 
and see that this real structure satisfies 
the assumptions in Theorem \ref{thm:introd_pure_polarized}. 
In Section \ref{sec:exampleP1tt*}, we calculate 
an example of $tt^*$-geometry for $X = \Proj^1$ 
with respect to the $\hGamma$-real structure. 
In this case, by Sabbah \cite{sabbah}, 
the $tt^*$-geometry is pure and polarized 
over the whole K\"{a}hler moduli space $H^2(\Proj^1,\C^*)$.  
We confirm Cecotti-Vafa's calculation 
\cite{cecotti-vafa-exactsigma} by a recursive 
Birkhoff factorization.

The convergence of the quantum cohomology 
is assumed throughout the paper. 
Also we consider only the even parity 
part of the cohomology, \emph{i.e.} 
$H^*(X)$ means  $\bigoplus_k H^{2k}(X)$. 
Note that the orbifold cohomology $H_{\rm orb}^*(\cX)$ 
is denoted also by $H_{\rm CR}^*(\cX)$ 
in the literature. 

This paper is a revision of part of the preprint 
\cite{iritani-realint-preprint} concerning real structures. 
The integral structure part of \cite{iritani-realint-preprint} 
was separated in \cite{iritani-Int}.

\vspace{5pt}  
\noindent  
{\bf Acknowledgments} 
Thanks are due to Tom Coates, Alessio Corti, Hsian-Hua Tseng 
for many useful discussions and their encouragement. 
This project is motivated by the joint work with them. 
The author is grateful to Martin Guest for 
useful discussions on the loop group and Iwasawa factorization. 
Guest also had the idea \cite{guest_durham, guest-qc_int} 
to incorporate real structures in quantum cohomology 
(independently).   
The author thanks Claus Hertling and the referees 
of the previous preprint \cite{iritani-realint-preprint}
for many useful comments. 
This research was supported by 
Inoue Research Award for Young Scientists, 
Grant-in-Aid for Young Scientists (B), 
19740039, 2007 
and EPSRC(EP/E022162/1).

\section{Real structures on \seminf VHS} 
\label{sec:generalVHS}
We introduce real (and integral) structures for 
a general semi-infinite variation of Hodge structure 
or \seminf VHS. 
We explain that a \seminf VHS with a real structure 
yields a Cecotti-Vafa structure when it is pure. 
A \seminf VHS was originally introduced by 
Barannikov \cite{barannikov-qpI, barannikov-proj}.  
A \seminf VHS with a real structure considered here corresponds to 
the TERP structure due to Hertling \cite{hertling-tt*} 
(see Remark \ref{rem:relation_hertling}). 
The exposition here largely follows 
the line of \cite{hertling-tt*, CIT:I}. 

\subsection{Definition}  
\label{subsec:realintstr_VHS_def} 
Let $\cM$ be a smooth complex analytic space 
and $\cO_\cM$ be the analytic structure sheaf on $\cM$. 
We introduce an additional complex plane 
$\C$ with co-ordinate $z$ and 
consider the product $\cM \times \C$. 
Let $\pi \colon \cM \times \C \to \cM$ 
be the projection. 
A \seminf VHS is a module over the push-forward 
$\pi_*\cO_{\cM\times \C}$ of the analytic structure sheaf
on $\cM\times \C$.  
Let $\Omega_\cM^1$ be the sheaf 
of holomorphic 1-forms on $\cM$ 
and $\Theta_\cM$ be the sheaf 
of holomorphic tangent vector fields on $\cM$.  

\begin{definition}[\cite{CIT:I}]  
A \emph{semi-infinite variation of Hodge structures}, 
or \seminf VHS is a locally free 
$\pi_*\cO_{\cM\times \C}$-module $\cF$ 
of rank $N$ endowed with a holomorphic flat connection 
\[ 
\nabla\colon \cF\rightarrow z^{-1} \cF \otimes \Omega^1_{\cM}
\]
and a perfect pairing 
\[
(\cdot,\cdot)_\cF\colon 
\cF\times \cF \to \pi_* \cO_{\cM\times \C}
\]
satisfying 
\begin{align*}
\nabla_X (f s) &= (Xf)s + f \nabla_X s, \\
[\nabla_X, \nabla_Y] s & = \nabla_{[X,Y]} s, \\  
(s_1, f(z)s_2)_{\cF} &= (f(-z)s_1,s_2)_{\cF} = f(z)(s_1,s_2)_{\cF},  \\
(s_1, s_2)_{\cF} & = (s_2, s_1)_{\cF}|_{z\to -z}, \\
X (s_1,s_2)_\cF &= (\nabla_X s_1,s_2)_\cF + (s_1,\nabla_X s_2)_\cF 
\end{align*} 
for sections $s,s_1,s_2$ of $\cF$, 
$f\in \pi_*\cO_{\cM\times \C}$ and $X, Y \in \Theta_\cM$. 
Here, $\nabla_X$ is a map from $\cF$ to $z^{-1}\cF$ 
and $z^{-1} \cF$ is regarded as a submodule 
of $\cF \otimes_{\pi_*\cO_{\cM\times \C}} \pi_*\cO_{\cM\times \C^*}$. 
The first two properties are part of the 
definition of a flat connection. 
The pairing $(\cdot,\cdot)_\cF$ 
is perfect in the sense that 
it induces an isomorphism of the fiber 
$\cF_\tau$ at $\tau\in \cM$ with 
$\Hom_{\cO(\C)}(\cF_\tau,\cO(\C))$, 
where $\cO(\C)$ is the space of holomorphic functions on $\C$. 

A {\it graded \seminf VHS} is a \seminf VHS $\cF$ endowed with a 
$\C$-endomorphism $\Grading \colon \cF \to \cF$ and 
an Euler vector field $E\in H^0(\cM,\Theta_\cM)$ satisfying 
\begin{align*}
\Grading(f s_1) &= (2 (z\partial_z + E) f ) s_1 +  f \Grading(s_1), \\
[\Grading, \nabla_X] &= \nabla_{2[E,X]}, \quad X\in \Theta_\cM, \\
2(z\partial_z+E) (s_1,s_2)_{\cF} 
&= (\Grading(s_1),s_2)_{\cF} + 
(s_1,\Grading(s_2))_{\cF} - 2n(s_1,s_2)_{\cF} 
\end{align*} 
where $n\in \C$. \qed 
\end{definition} 

A \seminf VHS is a semi-infinite analogue of the usual 
finite dimensional VHS without a real structure. 
The ``semi-infinite" flag
$\cdots \subset z \cF \subset \cF 
\subset z^{-1} \cF \subset z^{-2}\cF \subset \cdots$ 
plays the role of the Hodge filtration. 
The flat connection $\nabla_X$ shifts this filtration by one --- 
this is an analogue of the Griffiths transversality. 

The structure of a graded \seminf VHS $\cF$ 
can be rephrased in terms of a locally free sheaf 
$\cRz$ over $\cM\times \C$ with 
a flat connection $\hatnabla$. 
Here $\cRz$ is a locally free 
$\cO_{\cM\times \C}$-module of rank $N$  
such that $\cF = \pi_* \cRz$. 
We define the meromorphic connection\footnote
{The extended connection $\hatnabla$ over $\cM\times \C$ 
was denoted by $\nabla$ 
in the companion paper \cite{iritani-Int}.}
$\hatnabla$ on $\cRz$ 
\[
\hatnabla\colon 
\cRz \longrightarrow \frac{1}{z} \cRz \otimes 
\left
(\pi^* \Omega^1_{\cM} \oplus \cO_{\cM\times\C} \frac{dz}{z}
\right)  
\]
by the formula 
\begin{equation}
\label{eq:hatnabla}
\hatnabla s := \nabla s 
+ (\frac{1}{2} \Grading(s) - \nabla_E s - \frac{n}{2} s) 
\frac{dz}{z}    
\end{equation} 
for a section $s$ of $\cF = \pi_* \cRz$. 
It is easy to see that the conditions on 
$\Grading$ and $\nabla$ above imply that 
$\hatnabla$ is also flat.  
The pairing $(\cdot,\cdot)_{\cF}$ on $\cF$ 
induces a non-degenerate pairing on $\cRz$: 
\[
(\cdot,\cdot)_{\cRz} \colon 
(-)^*\cRz \otimes \cRz \rightarrow \cO_{\cM\times \C}, 
\]
where $(-)\colon \cM\times \C \to \cM\times \C$ is a map 
$(\tau,z) \mapsto (\tau,-z)$. 
This pairing is flat with respect to 
$\hatnabla$ on $\cRz$ and $(-)^*\hatnabla$ on $(-)^*\cRz$. 
Denote by $\cR$ the restriction of $\cRz$ to $\cM\times \C^*$. 
Since $\hatnabla$ is regular outside $z=0$, 
$\cR$ is a flat vector bundle over $\cM\times \C^*$.  
Let $R \to \cM\times \C^*$ be the 
$\C$-local system underlying the flat vector bundle $\cR$.  
This has a pairing $(\cdot,\cdot)_R \colon (-)^*R \otimes_\C R \to \C$ 
induced from $(\cdot,\cdot)_{\cRz}$.  
\begin{definition} 
\label{def:realintstr} 
Let $\cF$ be a graded \seminf VHS with $n\in \Z$. 
A \emph{real structure} on \seminf VHS 
is a sub $\R$-local system $R_\R\to \cM\times \C^*$ of $R$
such that $R = R_\R \oplus \iu R_\R$ 
and the pairing takes values in $\R$ on $R_\R$
\[
(\cdot,\cdot)_R \colon 
(-)^*R_\R \otimes_\R R_\R \rightarrow \R.  
\]
An \emph{integral structure} on \seminf VHS is 
a sub $\Z$-local system $R_\Z \to \cM\times \C^*$ of $R$   
such that $R= R_\Z \otimes_\Z \C$  
and the pairing takes values in $\Z$ on $R_\Z$
\[
(\cdot,\cdot)_R \colon 
(-)^* R_\Z \otimes R_\Z \rightarrow \Z 
\]
and is unimodular \emph{i.e.} induces an isomorphism 
$R_{\Z,(\tau,-z)} \cong \Hom(R_{\Z,(\tau,z)},\Z)$ for 
$(\tau,z)\in \cM\times \C^*$. \qed  
\end{definition} 
\begin{remark}
\label{rem:relation_hertling} 
A graded \seminf VHS with a real structure defined here 
is almost equivalent to a TERP($n$) structure 
introduced by Hertling \cite{hertling-tt*}. 
The only difference is that the flat connection 
$\hatnabla$ in TERP($n$) structure is not assumed 
to arise from a grading operator $\Grading$ 
and an Euler vector field $E$.  
Therefore, a graded \seminf VHS gives a TERP structure, 
but the converse is not true in general. 
For the convenience of the reader, we give  
differences in convention between \cite{hertling-tt*} and us. 
Let $\tilde{\nabla}$, $\tilde{R}$, $\tilde{R}_\R$, 
$\tilde{P}\colon \tilde{R} \otimes (-)^*\tilde{R} \to \C$ 
denote the flat connection, $\C$-local system, sub $\R$-local system  
and a pairing appearing in \cite{hertling-tt*}. 
They are related to our $\hatnabla$, 
$R$, $R_\R$, $(\cdot,\cdot)_{\cRz}$ as 
\begin{align*}
&\tilde\nabla = \hatnabla + \frac{n}{2} \frac{dz}{z}, \\ 
&\tilde{R} = (-z)^{-\frac{n}{2}} R, \quad 
\tilde{R}_{\R} = (-z)^{-\frac{n}{2}} R_\R,    \\ 
&\tilde{P}(s_1, s_2) =  z^n (s_2,s_1)_{\cRz}. 
\end{align*} 
Then $\tilde{R}$ is the local system defined by $\tilde{\nabla}$, 
$\tilde P$ is $\tilde{\nabla}$-flat and 
\[
\tilde{P}(\tilde{R}_{\R,(\tau,z)}\times \tilde{R}_{\R,(\tau,-z)}) 
= z^n \left(
z^{-n/2}R_{\R,(\tau,-z)}, (-z)^{-n/2} R_{\R,(\tau,z)}
\right)_R  
\subset  \iu^n \R.  
\]
\end{remark}

\subsection{Semi-infinite period map}
\label{subsec:semi-inf_period} 

\begin{definition}
For a graded \seminf VHS $\cF$,  
the \emph{spaces $\cH$, $\cV$ of multi-valued flat sections}  
are defined to be 
\begin{align*}
\cH &:= \{ s\in \Gamma(\widetilde{\cM} \times \C^*, \cR) \;;\; 
\nabla s =0 \}, \\
\cV &:= \{ s\in \Gamma((\cM\times \C^*)\sptilde,\cR) 
\;;\; \hatnabla s =0 \},  
\end{align*} 
where $\widetilde{\cM}$ and $(\cM\times \C^*)\sptilde$ 
are the universal 
covers of $\cM$ and $\cM\times \C^*$ respectively. 
The space $\cH$ is a free $\cO(\C^*)$-module, 
where $\cO(\C^*)$ 
is the space of holomorphic functions on $\C^*$.  
The space $\cV$ is a finite dimensional $\C$-vector space
identified with the fiber of the local system $R$. 
The flat connection $\hatnabla$ and the 
pairing $(\cdot,\cdot)_{\cRz}$ on $\cRz$ 
induce an operator 
\[
\hatnabla_{z\partial_z}\colon \cH\to \cH
\] 
and a pairing 
\[
(\cdot,\cdot)_{\cH}\colon \cH\times \cH\rightarrow \cO(\C^*)   
\]
satisfying 
\begin{align*}
(f(-z) s_1, s_2)_{\cH} &= (s_1, f(z) s_2)_{\cH} =
f(z)(s_1,s_2)_{\cH} \quad 
f(z) \in \cO(\C^*), \\
(s_1,s_2)_{\cH} &= (s_2,s_1)_{\cH}|_{z\mapsto -z} \\
z\partial_z(s_1,s_2)_{\cH}& = 
(\hatnabla_{z\partial_z} s_1, s_2)_{\cH} + 
(s_1,\hatnabla_{z\partial_z} s_2)_{\cH}. 
\end{align*} 
We regard the free $\cO(\C^*)$-module $\cH$ 
with the operator $\hatnabla_{z\partial_z}$ 
as a holomorphic flat vector bundle 
$(\sfH,\hatnabla_{z\partial_z})$ over $\C^*$: 
\begin{equation}
\label{eq:sfH}
\sfH \to \C^*,  \quad \cH = \Gamma(\C^*, \cO(\sfH)).  
\end{equation} 
Then $\cV$ can be identified with 
the space of multi-valued flat sections of $\sfH$. 
A pairing $(\cdot,\cdot)_{\cV}\colon \cV\otimes_\C \cV \to \C$ 
is defined by 
\begin{equation}
\label{eq:pairing_def_V}
(s_1,s_2)_{\cV} := (s_1(\tau,e^{\pi \iu} z), s_2(\tau,z))_{R}
\end{equation} 
where $s_1(\tau,e^{\pi \iu}z)\in \cR_{(\tau,-z)}$ 
denote the parallel translation of 
$s_1(\tau,z)\in \cR_{(\tau,z)}$ along the counterclockwise path 
$[0,1]\ni \theta \mapsto e^{\pi \iu \theta}z$.  \qed 
\end{definition} 

A \seminf VHS $\cF$ on $\cM$ defines a map 
from $\widetilde{\cM}$ to the 
\emph{Segal-Wilson Grassmannian} of $\cH$. 
For $u \in \cF_\tau$ at $\tau\in \widetilde{\cM}$, 
there exists a unique flat section $s_u\in \cH$ such that 
$s_u(\tau) =u$. 
This defines an embedding of a fiber $\cF_\tau$ into $\cH$:   
\begin{equation}
\label{eq:embedding_J}
\J_\tau \colon \cF_\tau \longrightarrow \cH, \quad 
u\longmapsto s_u, \quad 
\tau \in \widetilde{\cM}.  
\end{equation} 
We call the image $\F_\tau\subset \cH$ of this embedding 
the \emph{semi-infinite Hodge structure}. 
This is a free $\cO(\C)$-module of rank $N$. 
The family $\{\F_\tau\subset \cH\}_{\tau\in \widetilde{\cM}}$ 
of subspaces gives the 
\emph{moving subspace realization of \seminf VHS}. 
Fix a $\cO(\C^*)$-basis $e_1,\dots, e_N$ of $\cH$. 
Then the image of a local frame 
$s_1,\dots,s_N$ of $\cF$ over $\pi_*\cO_{\cM\times \C}$ 
under $\J_\tau$ can be written as 
$\J_\tau(s_j) = \sum_{i=1}^N e_i J_{ij}(\tau,z)$. 
When $z$ is restricted to $S^1=\{|z|=1\}$, 
the $N\times N$ matrix $(J_{ij}(\tau,z))$
defines an element of the smooth loop group 
$LGL_N(\C)=C^\infty(S^1,GL_N(\C))$. 
Another choice of a local basis of $\cF$ changes 
the matrix $(J_{ij}(\tau,z))$ by right multiplication by 
a matrix with entries in $\cO(\C)$. 
Thus the Hodge structure $\F_\tau$ gives a point 
$(J_{ij}(\tau,z))_{ij}$ in the smooth Segal-Wilson Grassmannian 
$\Gr_{\frac{\infty}{2}}(\cH) := LGL_N(\C)/L^+GL_N(\C)$ 
\cite{pressley-segal}. Here $L^+GL_N(\C)$ consists 
of smooth loops which are the boundary values 
of holomorphic maps $\{z\in \C\;;\; |z|<1\}\to GL_N(\C)$.  
The map 
\[
\widetilde{\cM} \ni \tau \longmapsto \F_\tau 
\in \Gr_{\frac{\infty}{2}}(\cH)
\]
is called the \emph{semi-infinite period map}. 
\begin{proposition}[{\cite[Proposition 2.9]{CIT:I}}]
\label{prop:property_VHS}
The semi-infinite period map $\tau\mapsto \F_\tau$ satisfies: 
\begin{itemize}
\item[(i)] $X \F_\tau \subset z^{-1} \F_\tau$ for $X\in \Theta_\cM$;   
\item[(ii)] $(\F_\tau,\F_\tau)_{\cH} \subset \cO(\C)$;  
\item[(iii)] $(\hatnabla_{z\partial_z}+E) \F_\tau 
\subset \F_\tau$. In particular, 
$\hatnabla_{z\partial_z}\F_\tau
\subset z^{-1}\F_\tau$.  
\end{itemize} 
The first property {\rm (ii)} 
is an analogue of Griffiths transversality 
and the second {\rm (iii)} is the Hodge-Riemann bilinear relation.
\end{proposition} 

In terms of the flat vector bundle $\sfH\to\C^*$ (\ref{eq:sfH})  
(such that $\cH = \Gamma(\C^*,\cO(\sfH))$), 
the Hodge structure $\F_\tau\subset \cH$ 
is considered to be an extension 
of $\sfH$ to $\C$ such that the flat connection has 
a pole of Poincar\'{e} rank 1 at $z=0$. 

Real and integral structures on \seminf VHS define the following  
subspaces $\cH_\R$, $\cV_\R$, $\cV_\Z$: 
\begin{align}
\label{eq:real_subspaces}
\begin{split}  
\cH_\R &:= \{ s\in \cH \; ;\; s(\tau,z) \in R_{\R,(\tau,z)} 
\text{ for } 
\tau\in \widetilde{\cM} \text{ and }  |z|=1 \} \\  
\cV_\R &:= \{ s\in \cV \; ;\; s(\tau,z) \in R_{\R,(\tau,z)} 
\text{ for } (\tau,z)\in (\cM\times \C^*)\sptilde \} \\
\cV_\Z &:= \{ s\in \cV \; ;\; s(\tau,z) \in R_{\Z,(\tau,z)}  
\text{ for } (\tau,z) \in (\cM\times \C^*)\sptilde \}   
\end{split} 
\end{align} 
Then $\cH_\R$ becomes a (not necessarily free) module 
over the ring $C^h(S^1,\R)$:
\begin{equation*}
C^h(S^1,\R):=\left\{
f(z) \in \cO(\C^*) \;;\; f(z)\in \R \text{ if } |z|=1 \right \}.   
\end{equation*} 
Note that we have $\cO(\C^*) = C^h(S^1,\R)\oplus \iu C^h(S^1,\R)$. 
The involution $\kappa$ on $\cO(\C^*)$ corresponding to the 
real form $C^h(S^1,\R)$ is given by  
\[
\kappa(f)(z) =\ov{f(\gamma(z))}, \quad f(z) \in \cO(\C^*),  
\]  
where $\gamma(z) = 1/\ov{z}$ and the $\ov{\phantom{A}}$ 
in the right-hand side 
is the complex conjugate. 
We also have $\cH \cong \cH_\R \oplus \iu \cH_\R$.   
This real form $\cH_\R\subset \cH$ defines 
an involution $\kappa_\cH\colon \cH\rightarrow \cH$ such that 
$\kappa_{\cH}(\alpha+\iu \beta) = \alpha - \iu \beta$ 
for $\alpha,\beta \in \cH_\R$. This satisfies 
\begin{align}
\label{eq:property_kappaH}
\begin{split}
\kappa_\cH (f s) &= \kappa(f)\kappa_{\cH}(s),  \\   
\kappa_\cH \hatnabla_{z\partial_z} &= - 
\hatnabla_{z\partial_z} \kappa_{\cH}, \\ 
\kappa((s_1,s_2)_{\cH}) &= (\kappa_{\cH}(s_1), 
\kappa_{\cH}(s_2))_{\cH}.   
\end{split}
\end{align} 
Note that $\kappa_{\cH}$ 
matches with the real involution on $R_{(\tau,z)}$ 
over the equator $\{|z|=1\}$. 
Similarly, we have $\cV= \cV_\R \oplus \iu \cV_\R$; 
we denote by $\kappa_{\cV}\colon \cV\to\cV$ 
the involution defined by the real structure $\cV_\R$.  

\begin{remark}
\label{rem:Cinfty} 
In the context of the smooth Grassmannian, 
it is more natural to work over $C^\infty(S^1,\C)$  
instead of $\cO(\C^*)$, where $S^1=\{|z|=1\}$. 
We put 
\[
\tcH := \cH \otimes_{\cO(\C^*)} C^\infty(S^1,\C), \quad 
\tbF_\tau := \F_\tau \otimes_{\cO(\C)} \cO(\D_0),  
\]
where $\cO(\D_0)$ is a subspace of $C^\infty(S^1,\C)$ consisting of 
functions which are the boundary values of holomorphic functions 
on the interior of the disc $\D_0=\{z\in \C\;;\; |z|\le 1\}$. 
The involution $\kappa_\cH\colon \tcH\to \tcH$ and 
the real form $\tcH_\R$ is defined similarly 
and the same properties hold. 
Conversely, using the flat connection 
$\hatnabla_{z\partial_z}$ in the $z$-direction, 
one can recover $\F_\tau$ from $\tbF_\tau$ 
since flat sections of $\hatnabla_{z\partial_z}$ 
determine an extension of the bundle on $\D_0$ to $\C$. 
\end{remark}

\subsection{Pure and polarized \seminf VHS}
Following Hertling \cite{hertling-tt*}, 
we define an extension $\hK$ of $\cRz$ across $z=\infty$.  
The properties ``pure and polarized" for $\cF$ 
are defined in terms of this extension.

\begin{definition}[Extension of $\cRz$ across $z=\infty$] 
\label{def:extension_infty}  
Let $\gamma\colon \cM\times \Proj^1 \rightarrow \cM\times \Proj^1$ 
be the map defined by $\gamma(\tau,z) = (\tau, 1/\ov{z})$. 
Let $\ov{\cM}$ denote the complex conjugate of $\cM$, \emph{i.e.} 
$\ov{\cM}$ is the same as $\cM$ as a real-analytic manifold 
but holomorphic functions on $\ov{\cM}$ are  
anti-holomorphic functions on $\cM$. 
The pull-back $\gamma^*\cRz$ of $\cRz$ 
has the structure of an 
$\cO_{\cM\times \ov{(\Proj^1\setminus\{0\})}}$-module. 
Thus its complex conjugate $\ov{\gamma^*\cRz}$ 
has the structure of 
an $\cO_{\ov{\cM}\times (\Proj^1\setminus \{0\})}$-module.  
Regarding $\cRz$ and $\ov{\gamma^*\cRz}$ 
as real-analytic vector bundles over $\cM\times \C$ 
and $\cM\times (\Proj^1\setminus\{0\})$, 
we glue them along $\cM\times \C^*$ by the fiberwise map  
\begin{equation}
\label{eq:gluingmap}
\begin{CD}
\cRz_{(\tau,z)} @>{\kappa}>> \ov{\cRz_{(\tau,z)}} 
@>{P(\gamma(z),z)}>> \ov{\cRz_{(\tau,\gamma(z))}} 
= \ov{\gamma^*\cRz_{(\tau,z)}},  
\end{CD} 
\quad z\in \C^*. 
\end{equation} 
Here the first map $\kappa$ is the real involution on 
$\cRz_{(\tau,z)}$ with respect to 
the real form $R_{\R,(\tau,z)}$ 
and the second map $P(\gamma(z),z)$ 
is the parallel translation for 
the flat connection $\hatnabla$ along the path 
$[0,1] \ni t \mapsto (1-t) z + t \gamma(z)$. 
Define $\hK\to \cM\times \Proj^1$ to 
be the real-analytic complex vector bundle 
obtained by gluing $\cRz$ and $\ov{\gamma^*\cRz}$ in this way. 
Notice that $\hK|_{\tau\times \Proj^1}$ 
has the structure of a holomorphic vector bundle 
since the gluing map (\ref{eq:gluingmap}) 
preserves the holomorphic structure 
in the $\Proj^1$-direction. \qed 
\end{definition} 

\begin{definition} 
\label{def:pure} 
A graded \seminf VHS $\cF$ with a real structure 
is called \emph{pure} at $\tau\in \cM$ if 
$\hK|_{\{\tau\}\times \Proj^1}$ 
is trivial as a holomorphic vector bundle on $\Proj^1$. 
\qed 
\end{definition} 

A pure graded \seminf VHS with a real structure here 
corresponds to the (trTERP) structure in \cite{hertling-tt*}. 
Here we follow the terminology in \cite{hertling-sevenheck}. 

We rephrase the purity in terms of the moving 
subspace realization
$\{\F_\tau \subset \cH\}$. 
When we identify $\cH$ with the space of global 
sections of $\hK|_{\{\tau\}\times \C^*} = 
\cR|_{\{\tau\}\times \C^*}$, 
it is easy to see that the involution 
$\kappa_\cH\colon \cH\to \cH$ is induced 
by the gluing map (\ref{eq:gluingmap}). 
Then $\F_\tau$ is identified with 
the space of holomorphic sections of $\hK|_{\{\tau\}\times \C^*}$ 
which can extend to $\{\tau\}\times \C$; 
$\kappa_\cH(\F_\tau)$ is identified with the space of holomorphic 
sections of $\hK|_{\{\tau\}\times \C^*}$ which can extend to 
$\{\tau\}\times (\Proj^1\setminus\{0\})$. 
Similarly, 
$\tbF_\tau$ (resp. $\kappa_{\cH}(\tbF_\tau)$) is 
identified with the space of smooth sections 
of $\hK|_{\{\tau\}\times S^1}$ which can extend 
to holomorphic sections on $\D_0$ (resp. $\D_\infty$), 
where $\D_0 = \{z\in \C\;;\; |z|\le 1\}$, 
$\D_\infty = \{z\in \C\cup \{\infty\} 
= \Proj^1 \;;\; |z| \ge 1\}$ 
and $\tbF_\tau$ 
is the space in Remark \ref{rem:Cinfty}.

\begin{proposition}
\label{prop:purity_condition} 
A graded \seminf VHS $\cF$ with a real structure 
is pure at $\tau\in \cM$ if and only if 
one of the following natural maps is an isomorphism: 
\begin{align} 
\label{eq:FcapkappaF}
\F_\tau \cap \kappa_{\cH}(\F_\tau) &\longrightarrow   \F_\tau/z\F_\tau, \\
\label{eq:FcapHreal} 
(\F_\tau \cap \cH_\R)\otimes \C &\longrightarrow  \F_\tau/z\F_\tau, \\ 
\label{eq:FpluskappaF} 
\F_\tau \oplus z^{-1} \kappa_{\cH}(\F_\tau) & \longrightarrow  \cH. 
\end{align} 
This holds also true when $\F_\tau$, $\cH$, $\cH_\R$ are replaced with 
$\tbF_\tau$, $\tcH$, $\tcH_\R$ in Remark \ref{rem:Cinfty}. 
When $\cF$ is pure at some $\tau$, 
$\cH_\R$ is a free module over $C^h(S^1,\R)$.  
\end{proposition} 
\begin{proof}
Under the identifications we explained above, 
$\F_\tau\cap \kappa_{\cH}(\F_\tau)$ is 
identified with the space of global sections of 
$\hK|_{\{\tau \}\times \Proj^1}$ 
and the natural map 
$\F_\tau \cap \kappa_{\cH}(\F_\tau) \to \F_\tau/z\F_\tau$ 
corresponds to the restriction to $z=0$ 
(note that $\F_\tau/z\F_\tau \cong \hK_{(\tau,0)}$). 
Therefore (\ref{eq:FcapkappaF}) is an isomorphism 
if and only if $K|_{\{\tau \}\times \Proj^1}$ is trivial. 
$\F_\tau\cap \kappa_\cH(\F_\tau)$ 
is invariant under $\kappa_{\cH}$ 
and its real form is given by $\F_\tau\cap \cH_\R$. 
Therefore, we have 
$\F_\tau\cap \kappa_{\cH}(\F_\tau) 
\cong (\F_\tau\cap \cH_\R)\otimes \C$. 
Thus (\ref{eq:FcapkappaF}) is an isomorphism if and only if 
so is (\ref{eq:FcapHreal}). 
Similarly, we can see that 
(\ref{eq:FpluskappaF}) is an isomorphism 
if $\hK|_{\{\tau \}\times \Proj^1}$ is trivial. 
Conversely, we show that (\ref{eq:FcapkappaF}) is an isomorphism 
if so is (\ref{eq:FpluskappaF}).   
The injectivity of the map 
$\F_\tau \cap \kappa_{\cH}(\F_\tau) \to \F_\tau/z\F_\tau$ 
is easy to check. 
Take $v\in \F_\tau$. By assumption, $z^{-1}v= v_1 + v_2$ for some 
$v_1 \in \F_\tau$ and $v_2 \in z^{-1}\kappa_{\cH}(\F_\tau)$. 
Thus $v-zv_1 = z v_2 \in \F_\tau\cap \kappa_{\cH}(\F_\tau)$ and 
the image of this element in $\F_\tau/z\F_\tau$ is $[v]$.  
The discussion on the spaces 
$\tbF_\tau$, $\tcH$ and $\tcH_\R$ are similar.

The last statement: Since $\F_\tau\cap \kappa_{\cH}(\F_\tau) \cong 
(\F_\tau \cap \cH_\R)\otimes \C$, we can take a global basis 
of the trivial bundle $\hK|_{\{\tau \}\times \Proj^1}$ 
from $\F_\tau\cap \cH_\R$. The module $\cH_\R$ is freely 
generated by such a basis over $C^h(S^1,\R)$. 
\end{proof}

\begin{definition} 
\label{def:polarized} 
A graded \seminf VHS $\cF$ 
with a real structure is 
called \emph{polarized} at $\tau\in \cM$ 
if the Hermitian pairing $h$ on 
$\F_\tau \cap \kappa_{\cH}(\F_\tau) 
\cong \Gamma(\Proj^1,\hK|_{\{\tau\}\times \Proj^1})$ 
defined by 
\[
h\colon s_1 \times s_2 \longmapsto  (\kappa_{\cH}(s_1), s_2)_{\cH} 
\]
is positive definite. Note that this pairing 
takes values in $\C$ since 
$(\F_\tau,\F_\tau)_{\cH}\subset \cO(\C)$ and  
$(\kappa_{\cH}(\F_\tau),\kappa_\cH(\F_\tau))_{\cH} 
\subset \cO(\Proj^1\setminus\{0\})$ by (\ref{eq:property_kappaH}). 
It is easy to show that a polarized \seminf VHS is 
necessarily pure at the same point. \qed
\end{definition}

\begin{remark} 
\label{rem:Birkhoff_Iwasawa} 
In order to obtain a basis of 
$\tbF_\tau\cap \kappa_{\cH}(\tbF_\tau)$ 
or $\tbF_\tau \cap \tcH_\R$,  
we can make use of Birkhoff or Iwasawa factorization.  
Take an $\cO(\D_0)$-basis $s_1,\dots,s_N$ of $\tbF_\tau$. 
Define an element $A(z)=(A_{ij}(z))$ 
of the loop group $LGL_N(\C)$ by 
\[
[\kappa_\cH(s_1),\dots,\kappa_\cH(s_N)] = [s_1,\dots,s_N] A(z), 
\quad \text{\emph{i.e.} } \kappa_\cH(s_i) = \sum_{j} s_j A_{ji}(z). 
\]
If $A(z)$ admits the Birkhoff factorization 
$A(z) = B(z) C(z)$, 
where $B(z)$ and $C(z)$ are holomorphic maps 
$B(z)\colon \D_0 \to GL_N(\C)$,  
$C(z)\colon \D_{\infty} \to GL_N(\C)$ such that $B(0)=\unit$, 
then we obtain a $\C$-basis of 
$\tbF_\tau \cap \kappa_{\cH}(\tbF_\tau)$ as 
\begin{equation}
\label{eq:A=BC}
[\kappa_\cH(s_1),\dots,\kappa_\cH(s_N)] C(z)^{-1} 
= [s_1,\dots,s_N] B(z).   
\end{equation} 
Here, $\cF$ is pure at $\tau\in \cM$ if and only if 
$A(z)$ admits the Birkhoff factorization, 
\emph{i.e.} $A(z)$ is in the ``big cell" 
of the loop group. 
In particular, the purity is 
an open condition for $\tau \in \cM$. 
On the other hand, the Iwasawa-type 
factorization appears as follows. 
Assume that we have a basis $e_1,\dots,e_N$ 
of $\tcH_\R$ over $C^\infty(S^1,\R)$ such that 
$(e_i,e_j)_{\tcH} =\delta_{ij}$ 
and a basis $s_1,\dots,s_N$ of $\tbF_\tau$ over $\cO(\D_0)$ 
such that $(s_i,s_j)_{\tcH}= \delta_{ij}$. 
Define a matrix $J(z)$ by 
\[
[s_1,\dots, s_N] = [e_1,\dots,e_N] J(z).  
\]
This $J(z)$ lies in the {\it twisted loop group} 
$LGL_N(\C)_{\rm tw}$: 
\[
LGL_N(\C)_{\rm tw} := 
\{J \colon S^1 \to GL_N(\C)\;;\;
J(-z)^{\rm T} J(z) = \unit \}.   
\]
If $J(z)$ admits an 
Iwasawa-type factorization $J(z)=U(z)B(z)$, where 
$U\colon S^1 \to GL_N(\R)$ with 
$U(-z)^{\rm T} U(z) = \unit$ and  
$B\colon \D_0 \to GL_N(\C)$ with $B(-z)^{\rm T} B(z) =\unit$, 
then we obtain an $\R$-basis of $\tbF_\tau \cap \tcH_\R$ as 
\[
[s_1,\dots,s_N] B(z)^{-1} = 
[e_1,\dots,e_N] U(z)    
\]
which is orthonormal with respect to $(\cdot,\cdot)_{\tcH}$. 
In this case, the pairing $(\cdot,\cdot)_{\tcH}$ 
restricted to $\tbF_\tau \cap \tcH_\R$ is an 
$\R$-valued \emph{positive definite}  
symmetric form. 
The map $\tau \mapsto J(z)$ gives rise to 
the semi-infinite period map in 
Section \ref{subsec:semi-inf_period}: 
\[
\cM \ni \tau \longmapsto [J(z)] \in LGL_N(\C)_{\rm tw}/ 
LGL^+_N(\C)_{\rm tw}.  
\]
Here, $\cF$ is pure at $\tau$ 
and $(\tbF_\tau \cap \tcH_\R, (\cdot,\cdot)_{\tcH})$ is 
positive definite if and only if the image of this map 
lies in the $LGL_N(\R)_{\rm tw}$-orbit of $[\unit]$. 
This orbit is open, but not dense.   
We owe the Lie group theoretic viewpoint here 
to Guest \cite{guest_durham, guest-qc_int}. 
\end{remark}

\begin{remark} 
In addition to the purity and the polarization, 
Hertling-Sevenheck \cite{hertling-sevenheck} 
and Katzarkov-Kontsevich-Pantev \cite{KKP} 
considered the compatibility of a real (or rational) 
structure and the Stokes structure. 
\end{remark}

\subsection{Cecotti-Vafa structure} 
We describe the Cecotti-Vafa structure ($tt^*$-geometry) 
associated to a pure graded \seminf VHS with a real structure. 

Define a complex vector bundle $K\to\cM$ by 
$K:= \hK|_{\cM\times \{0\}}$. 
This is the real analytic vector bundle 
underlying $\cF/z\cF \cong \cRz|_{\cM\times \{0\}}$. 
Let $\cA_{\cM}^p$ be the sheaf of complex-valued 
$C^\infty$ $p$-forms on $\cM$ and 
$\cA_{\cM}^1 = \cA_{\cM}^{1,0} \oplus \cA_{\cM}^{0,1}$ 
be the type decomposition.

\begin{proposition}[{\cite[Theorem 2.19]{hertling-tt*}}]
\label{prop:CV-str} 
Assume that a graded \seminf VHS $\cF$ with a real structure 
is pure over $\cM$. 
Then the vector bundle $K$ is equipped 
with a Cecotti-Vafa structure 
$(\kappa, g, C,\tC, D, \cQ, \cU,\ov\cU)$. 
This is given by the data (see (\ref{eq:defofkappa}), 
(\ref{eq:defofg}), (\ref{eq:defofCCD}), (\ref{eq:defofUUQ})):  
\begin{itemize} 
\item A complex-antilinear involution $\kappa \colon K_\tau \to K_\tau$; 
\item A non-degenerate, symmetric, $\C$-bilinear metric 
$g\colon K_\tau \times K_\tau \to \C$ 
which is real with respect to $\kappa$, i.e. 
$g(\kappa u_1, \kappa u_2) = \ov{g(u_1,u_2)}$;   
\item Endomorphisms $C \in \End(K)\otimes \cA^{1,0}_\cM$, 
$\tC\in \End(K)\otimes \cA^{0,1}_\cM$ such that 
$\tC_{\oi} = \kappa C_i \kappa$;  
\item A connection $D\colon K\to K\otimes \cA^1_\cM$ 
real with respect to $\kappa$, i.e.
$D_{\oi} = \kappa D_i \kappa$; 
\item Endomorphisms $\cQ,\cU,\ov\cU\in \End(K)$ 
such that $\cU = C_E$, $\ov\cU=\kappa \cU \kappa = \tC_{\ov{E}}$ 
and $\cQ \kappa = -\kappa \cQ$   
\end{itemize}  
satisfying the integrability conditions 
\begin{align*}
&[D_i,D_j]=0, \quad D_iC_j - D_j C_i=0, \quad [C_i,C_j]=0, \\ 
&[D_\oi,D_\oj]=0, \quad  D_\oi \tC_\oj - D_\oj \tC_\oi=0, \quad 
[\tC_\oi, \tC_\oj]=0, \\
& D_i \tC_\oj =0, \quad D_\oi C_j =0,\quad  
[D_i,D_\oj] + [C_i,\tC_\oj]=0, \\ 
&D_i\ov\cU=0, \quad D_i \cQ - [\ov\cU, C_i]=0, 
\quad D_i\cU - C_i + [\cQ,C_i]=0, \quad [\cU_,C_i]=0, \\  
&D_\oi\cU=0, \quad D_\oi \cQ + [\cU, \tC_\oi]=0, \quad  
D_\oi\ov\cU - \tC_\oi - [\cQ,\tC_\oi]=0, \quad [\ov\cU,\tC_\oi]=0, 
\end{align*} 
and the compatibility with the metric 
\begin{align*}
&\partial_i g(u_1,u_2) = g(D_iu_1,u_2)+g(u_1,D_iu_2), \\
&\partial_\oi g(u_1,u_2) = g(D_\oi u_1, u_2) + g(u_1, D_\oi u_2), \\
& g(C_i u_1, u_2) = g(u_1, C_iu_2),\quad 
g(\tC_\oi u_1, u_2) = g(u_1, \tC_\oi u_2), \\  
& g(\cU u_1, u_2) = g(u_1, \cU u_2), \quad 
g(\ov\cU u_1, u_2) = g(u_1, \ov\cU u_2), \\ 
& g(\cQ u_1, u_2)+g(u_1,\cQ u_2)=0.   
\end{align*} 
Here we chose a local complex co-ordinate system $\{t^i\}$
on $\cM$ and used the notation 
$D_i = D_{\partial/\partial t^i}$, 
$D_\oi = D_{\partial/\partial \ov{t^i}}$, etc. 
The Hermitian metric $h$ 
in Definition \ref{def:polarized} 
is related to $g$ by  
\[
h(u_1,u_2) = g(\kappa(u_1),u_2).  
\]
\end{proposition} 

A concrete example of the Cecotti-Vafa structure 
will be given in Section \ref{sec:exampleP1tt*}. 
We explain the construction of the above data 
from the \seminf VHS $\cF$. 
Because $\cF$ is pure, we have a canonical identification 
\[
\Phi_\tau \colon K_\tau \cong \Gamma(\Proj^1,\hK|_{\{\tau\}\times \Proj^1}) 
\cong \F_\tau \cap \kappa_{\cH}(\F_\tau).  
\]
The involution $\kappa_\cH$ and the pairing $(\cdot,\cdot)_\cH$ 
restricted to $\F_\tau \cap \kappa_{\cH}(\F_\tau)$ induce 
an involution $\kappa$ and a $\C$-bilinear pairing 
$g$ on $K_\tau$: 
\begin{align}
\label{eq:defofkappa}
\Phi_\tau(\kappa(u)) & := \kappa_\cH(\Phi_\tau(u)), \\
\label{eq:defofg} 
g(u_1,u_2) & := (\Phi_\tau(u_1),\Phi_\tau(u_2))_{\cH} 
\in \C 
\end{align} 
satisfying 
\[
g(\kappa u_1, \kappa u_2 ) = \ov{g(u_1,u_2)}, 
\quad g(u_1,u_2) = g(u_2,u_1).  
\]
Note that the subspace $\F_\tau\cap \kappa_{\cH}(\F_\tau)$ 
depends on the parameter $\tau$ real analytically.  
A $C^\infty$-version of the 
Griffiths transversality gives 
\begin{align*} 
&X^{(1,0)} \F_\tau \subset z^{-1} \F_\tau, 
& & X^{(0,1)} \F_\tau \subset \F_\tau, \\
&X^{(1,0)} \kappa_\cH(\F_\tau) \subset \kappa_\cH(\F_\tau), & 
&X^{(0,1)} \kappa_\cH(\F_\tau) \subset z \kappa_\cH(\F_\tau),   
\end{align*} 
where $X^{(1,0)}\in T^{1,0}_\tau\cM$ and 
$X^{(0,1)}\in T^{0,1}_\tau \cM$.  
For $X^{(1,0)}\in T_\tau^{1,0}\cM$, we have 
\[
X^{(1,0)}(\F_\tau \cap \kappa_\cH(\F_\tau) )
\subset z^{-1} \F_\tau \cap \kappa_{\cH}(\F_\tau) =  
z^{-1}(\F_\tau \cap \kappa_\cH(\F_\tau)) \oplus 
(\F_\tau \cap \kappa_\cH(\F_\tau)). 
\] 
Similarly for $X^{(0,1)}\in T_\tau^{(0,1)}\cM$, we have  
\begin{align*}
X^{(0,1)} (\F_\tau \cap \kappa_\cH(\F_\tau)) 
\subset (\F_\tau \cap \kappa_\cH(\F_\tau)) \oplus 
z (\F_\tau \cap \kappa_\cH(\F_\tau)). 
\end{align*} 
Hence we can define endomorphisms 
$C\colon K \to K\otimes \cA^{1,0}$, 
$\tC\colon K \to K\otimes \cA^{0,1}$,  
and a connection $D \colon K\to K\otimes \cA^1$ by 
\begin{equation}
\label{eq:defofCCD}
X\Phi_\tau(u_\tau) 
= z^{-1} \Phi_\tau(C_X(u_\tau)) + \Phi_\tau(D_X(u_\tau)) 
+ z \Phi_\tau(\tC_X(u_\tau))
\end{equation} 
for a section $u_\tau$ of $K$. 
By applying $\kappa_\cH$ on the both hand sides, 
\[
\ov{X} \Phi_\tau(\kappa u_\tau) 
= z^{-1} \Phi_\tau ( \kappa \tC_X(u_\tau))
+ \Phi_\tau(\kappa D_X(u_\tau)) + 
z \Phi_\tau (\kappa C_X(u_\tau)). 
\]
Therefore, we must have 
\[
C_{\ov{X}} \kappa = \kappa  \tC_X, \quad  
\kappa   D_X = D_{\ov{X}}  \kappa, \quad 
X\in T\cM \otimes_\R \C.   
\] 
Similarly, we can define endomorphisms 
$\cU,\ov{\cU},\cQ \colon K\to K$ by 
\begin{align}
\label{eq:defofUUQ} 
\hatnabla_{z\partial_z} \Phi_\tau(u_\tau) 
= -z^{-1} \Phi_\tau(\cU(u_\tau)) +  
\Phi_\tau(\cQ(u_\tau)) + z \Phi_\tau(\ov{\cU}(u_\tau)). 
\end{align} 
Because $\hatnabla_{z\partial_z}$ is purely imaginary 
(\ref{eq:property_kappaH}), we have 
\[ 
\kappa \cQ = -\cQ  \kappa, \quad 
\ov{\cU}=\kappa \cU \kappa.
\] 
By $(\hatnabla_{z\partial_z} + E)\F_\tau \subset \F_\tau$
in Proposition \ref{prop:property_VHS}, we find 
\[
\cU = C_E, \quad \ov\cU = \tC_{\ov{E}}. 
\]
We have a canonical isomorphism 
\[
\pi^* K \cong \hK, \quad \text{where } \ 
\pi\colon \cM\times \Proj^1 \to \cM.  
\]
Let $C^{\infty h}(\pi^*K)$ be the sheaf of $C^\infty$ 
sections of $\pi^*K\cong \hK$ which are holomorphic 
on each fiber $\{\tau\}\times \Proj^1$. 
Under the isomorphism above, the flat connection 
$\hatnabla$ on $\cRz = \hK|_{\cM\times \C}$ 
can be written in the form:   
\begin{align}
\nonumber 
\hatnabla  \colon C^{\infty h}(\pi^*K) \longrightarrow 
C^{\infty h}(\pi^*K)\otimes
\Bigl( & z^{-1}\cA^{1,0}_{\cM} \oplus \cA^{1}_{\cM}
\oplus z \cA^{0,1}_{\cM}   \\ 
\nonumber 
 & \oplus (z^{-1}\cA^0_{\cM} \oplus \cA^0_{\cM} \oplus z\cA^0_{\cM})
\frac{dz}{z} \Bigr) \\
\label{eq:hatnabla_K} 
\hatnabla = z^{-1}C + D + z \tC  
 \,+\, &(z\partial_z -z^{-1}\cU + \cQ + z\ov{\cU})\otimes \frac{dz}{z}.   
\end{align} 
Under the same isomorphism, 
the pairing $(\cdot,\cdot)_{\cRz}$ on $\cRz= \hK|_{\cM\times \C}$ 
can be written as  
\begin{align*}
C^{\infty h}((-)^*(\pi^*K)) 
\otimes 
C^{\infty h}(\pi^*K) & \to C^{\infty h}(\cM\times \Proj^1) \\
s_1(\tau,-z) \otimes   s_2(\tau,z) & \longmapsto 
g(s_1(\tau,-z),s_2(\tau,z)).  
\end{align*} 
Unpacking the flatness of $\hatnabla$ and 
$\hatnabla$-flatness of the pairing in terms of 
$C,\tC,D,\cU,\cQ$ and $g$,  
we arrive at the equations in Proposition \ref{prop:CV-str}.

\begin{remark}
(i) The $(0,1)$-part $\hatnabla_{\oi} = D_\oi + z \tC_\oi$ 
of the flat connection (\ref{eq:hatnabla_K}) 
gives the holomorphic structure on 
$\hK|_{\cM\times \{z\}}$ 
which corresponds to the holomorphic structure on $\cRz$. 
In particular, $D$ is identified with  
the canonical connection associated to the Hermitian metric $h$ 
on the holomorphic vector bundle $\cF/z\cF$. 
Similarly, the $(1,0)$-part $D_i + z^{-1} C_i$ 
gives an anti-holomorphic structure 
on $\hK|_{\cM\times \{z\}}$  
which corresponds to the anti-holomorphic structure on 
$\ov{\gamma^* \cRz}$. 

(ii) Among the data of the Cecotti-Vafa structure, 
one can define the data ($C$, $D_E +\cQ$, $\cU$, $g$) 
without choosing a real structure. 
In fact, $C_X$ is given by the map 
$\cF/z\cF \ni [s] \mapsto  [z\nabla_X s] \in \cF/z\cF$, 
$D_E + \cQ$ is given by the map 
$\cF/z\cF \ni [s] \mapsto [\frac{1}{2}(\Grading - n) s] 
\in \cF/z\cF$, 
$\cU = C_E$, and $g$ is given by 
$g([s_1],[s_2]) = (s_1,s_2)_{\cF}|_{z=0}$ for $s_i\in \cF$.  
In the case of quantum cohomology, 
$C_i$ is the quantum multiplication 
$\phi_i\circ$ by some $\phi_i\in H^*_{\rm orb}(\cX)$
(see (\ref{eq:quantumproduct_divisor}), 
(\ref{eq:Dubrovinconn})) and 
$g$ is the Poincar\'{e} pairing. 
\end{remark}

\begin{remark} 
A Frobenius manifold structure \cite{dubrovin-2D} on $\cM$ 
arises from a miniversal \seminf VHS
(in the sense of \cite[Definition 2.8]{CIT:I}) 
without a real structure. 
To obtain a Frobenius manifold structure, 
we need a choice of an opposite subspace 
$\cH_-\subset \cH$: 
a sub free $\cO(\Proj^1\setminus\{0\})$-module $\cH_-$ of $\cH$ 
satisfying  
\[
\cH = \F_\tau \oplus \cH_-, \quad 
\hatnabla_{z\partial_z}\cH_- \subset \cH_-.  
\] 
The choice of $\cH_-$ corresponds to giving 
a logarithmic extension of the flat vector bundle 
$(\sfH,\hatnabla_{z\partial_z})$ at $z=\infty$. 
A graded \seminf VHS with the choice of an opposite subspace 
corresponds to the (trTLEP)-structure in 
Hertling \cite{hertling-tt*}. 
See \cite{barannikov-qpI,hertling-tt*, CIT:I} 
for the construction of Frobenius manifolds 
from this viewpoint.   
In the $tt^*$-geometry, 
the complex conjugate $\kappa_{\cH}(\F_\tau)$ 
of the Hodge structure $\F_\tau$ 
plays the role of the opposite subspace 
(see (\ref{eq:FpluskappaF})). 
When a miniversal \seminf VHS is equipped 
with both a real structure and an opposite subspace, 
under certain conditions, 
$\cM$ has a CDV (Cecotti-Dubrovin-Vafa) structure,  
which dominates both 
Frobenius manifold structure and 
Cecotti-Vafa structure on $T\cM$. 
See \cite[Theorem 5.15]{hertling-tt*} 
for more details. 
\end{remark}

\section{Real structures on the quantum cohomology}
\label{sec:A-model} 

In this section, we give a review of orbifold 
quantum cohomology and introduce a real structure on it. 
Some of the basic materials here have overlaps with 
the companion paper \cite{iritani-Int} 
and we refer the reader to it for the proofs.  

\subsection{Orbifold quantum cohomology} 
Quantum cohomology for orbifolds have been developed by 
Chen-Ruan \cite{chen-ruan:GW} for symplectic orbifolds 
and Abramovich-Graber-Vistoli \cite{AGV} for 
smooth Deligne-Mumford stacks. 
Real structures make sense for both (symplectic and algebraic) 
categories, but we will work in the algebraic category. 
For example, we need the Lefschetz decomposition 
in the proof of Theorem \ref{thm:pure_polarized}.

Let $\cX$ be a proper smooth Deligne-Mumford stack over $\C$.  
Let $I\cX$ be the \emph{inertia stack} of $\cX$, 
which is defined to be the fiber product 
$\cX\times_{\cX\times \cX} \cX$ of the two 
diagonal morphisms $\Delta\colon \cX\to \cX\times \cX$. 
A point of $I\cX$ is given by a pair $(x,g)$ of 
a point $x\in \cX$ and $g\in \Aut(x)$. 
Here $g$ is called the \emph{stabilizer} at $(x,g)\in I\cX$. 
The inertia stack is decomposed into connected components: 
\[
I\cX = \bigsqcup_{v\in \sfT} \cX_v = 
\cX_0 \cup \bigsqcup_{v\in \sfT'} \cX_v, \quad 
\cX_0 = \cX. 
\] 
Here $\sfT$ is the index set of connected components,  
$0 \in \sfT$ corresponds to the distinguished component 
with the trivial stabilizer and 
$\sfT' = \sfT \setminus \{0\}$. 
For each connected component $\cX_v$ of $I\cX$, 
we associate a rational number $\iota_v$ called \emph{age}. 
For $(x,g) \in \cX_v \subset I\cX$, 
let $0\le f_1,\dots,f_n<1$ ($n=\dim_\C \cX$) 
be rational numbers such that 
the stabilizer $g$ acts on the tangent space 
$T_x \cX$ with eigenvalues 
$\exp(2\pi\iu f_1),\dots, \exp(2\pi\iu f_n)$ 
(with multiplicities).  
Then we set 
\[
\iota_v := f_1 + \cdots + f_n. 
\]
The \emph{(even parity) orbifold cohomology group} 
$H_{\rm orb}^*(\cX)$ is defined to be 
\[
H_{\rm orb}^k(\cX) = 
\bigoplus_{v\in \sfT : k-2\iota_v \in 2\Z} 
H^{k-2\iota_v}(\cX_v,\C).  
\]
The degree $k$ of the orbifold cohomology 
can be a fractional number in general. 
Each factor $H^*(\cX_v,\C)$ in the right-hand side 
denotes the cohomology group of $\cX_v$ as a topological space.  
We define an involution $\inv \colon I\cX \to I\cX$ by 
$\inv(x,g) = (x,g^{-1})$ 
and the \emph{orbifold Poincar\'{e} pairing} by  
\[
(\alpha, \beta)_{\rm orb} : = 
\int_{I\cX} \alpha \cup \inv^*(\beta) 
= \sum_{v\in \sfT} \int_{\cX_v} 
\alpha_v \cup \beta_{\inv(v)}.  
\]
where $\alpha_v$, $\beta_v$ are the $\cX_v$-components 
of $\alpha$, $\beta$ and $\inv \colon \sfT \to \sfT$ 
denotes the induced involution on $\sfT$.  
This is a symmetric non-degenerate pairing of 
degree $-2n$, where $n=\dim_\C\cX$. 

Now assume that the coarse moduli space of $\cX$ is projective.  
The \emph{genus zero orbifold Gromov-Witten invariants} 
are integrals of the form: 
\begin{equation}
\label{eq:GWcorrelator}
\corr{\alpha_1\psi^{k_1},\dots,\alpha_l \psi^{k_l}}_{0,l,d}^\cX 
= \int_{[\cX_{0,l,d}]^{\rm vir}} \prod_{i=1}^l \ev_i^*(\alpha_i) \psi_i^{k_i}
\end{equation} 
for $\alpha_i \in H_{\rm orb}^*(\cX)$, $d\in H_2(\cX,\Q)$ 
and non-negative integers $k_i$.  
Here $\cX_{0,l,d}$ is the moduli space of 
(balanced twisted) stable maps to $\cX$ 
of degree $d$ and with $l$ marked points and 
$[\cX_{0,l,d}]$ is its virtual fundamental class. 
The map $\ev_i\colon \cX_{0,l,d} \to I\cX$ 
is the evaluation map\footnote
{The map $\ev_i$ only exists as a map of topological spaces. 
In \cite{AGV}, $\ev_i$ takes values in the 
\emph{rigidified inertia stack} 
which is the same as $I\cX$ as a topological space 
but is different as a stack.} at the $i$-th marked point 
and $\psi_i$ is the first Chern class of the 
line bundle over $\cX_{0,l,d}$ whose fiber at a stable map 
is the cotangent space of the coarse curve 
at the $i$-th marked point. See \cite{AGV} for details. 
The Gromov-Witten invariants (\ref{eq:GWcorrelator}) 
are non-zero only when $d$ is in the semigroup 
$\Eff_\cX \subset H_2(\cX,\Q)$ generated by 
effective curves.

The \emph{orbifold quantum cohomology} is a formal 
family of associative and commutative products $\bullet_\tau$ 
on $H_{\rm orb}^*(\cX) \otimes \C[\![\Eff_\cX]\!]$ 
parametrized by $\tau\in H_{\rm orb}^*(\cX)$.   
It is defined by the formula 
\begin{equation*} 
(\alpha \bullet_\tau \beta,\gamma)_{\rm orb} 
= \sum_{d\in \Eff_\cX} \sum_{l\ge 0} 
\sum_{k=1}^N \frac{1}{l!} 
\corr{\alpha,\beta,\tau,\dots,\tau, \gamma}_{0,l+3,d}^\cX Q^d,   
\end{equation*} 
where $Q^d$ is an element of the group ring $\C[\Eff_\cX]$ 
corresponding to $d\in \Eff_\cX$. 
We decompose the parameter $\tau$ as 
\begin{equation}
\label{eq:tau_decomp} 
\tau = \tau_{0,2} + \tau', \quad 
\tau_{0,2}\in H^2(\cX), \quad   
\tau' \in \bigoplus_{k\neq 1} H^{2k}(\cX) \oplus 
\bigoplus_{v\in \sfT'} H^*(\cX_v). 
\end{equation} 
Using the divisor equation \cite{tseng:QRR, AGV}, 
we have 
\begin{align} 
\label{eq:quantumproduct_divisor}
(\alpha \bullet_\tau \beta,\gamma)_{\rm orb} 
= \sum_{d\in \Eff_\cX} \sum_{l\ge 0} 
\sum_{k=1}^N 
\frac{1}{l!} 
\corr{\alpha,\beta,\tau',\dots,\tau',\gamma}_{0,l+3,d}^{\cX} 
e^{\pair{\tau_{0,2}}{d}} Q^d. 
\end{align} 
Thus the quantum product is a formal power series in 
$e^{\tau_{0,2}}Q$ and $\tau'$. 
\begin{assumption} 
\label{assump:convergence} 
The specialization $Q=1$ of the quantum product $\bullet_\tau$ 
\[
\circ_\tau := \bullet_\tau|_{Q=1} 
\]
is convergent over a connected, simply connected 
open set $U \subset H^*_{\rm orb}(\cX)$ 
containing the set 
\[
\left\{\tau \in H^*_{\rm orb}(\cX) \; ; 
\; \Re\pair{\tau_{0,2}}{d} < -M, \forall d\in 
\Eff_\cX \setminus \{0\}, \ \|\tau'\| < 1/M
\right\}  
\]
for a sufficiently big $M>0$. Here we used 
the decomposition (\ref{eq:tau_decomp}) 
and $\|\cdot\|$ is some norm on $H^*_{\rm orb}(\cX)$. 
\end{assumption} 

Under this assumption, $(H^*_{\rm orb}(\cX),\circ_\tau)$ 
defines an analytic family of rings over $U$. 
The domain $U$ here contains the following limit direction:  
\begin{equation}
\label{eq:largeradiuslimit}
\Re\pair{\tau_{0,2}}{d} \to -\infty, \quad 
\forall d\in \Eff_\cX\setminus\{0\}, 
\quad \tau' \to 0.  
\end{equation} 
This is called the \emph{large radius limit}. 
In this limit, $\circ_\tau$ goes 
to the orbifold cup product 
due to Chen-Ruan \cite{chen-ruan:new_coh_orb} 
(which is the same as the cup product 
when $\cX$ is a manifold). 

\subsection{Quantum cohomology \seminf VHS and 
the Galois action}
\label{subsec:AmodelVHS} 
Take a homogeneous basis $\{\phi_i\}_{i=1}^N$  
of $H^*_{\rm orb}(\cX)$. 
Let $\{t^i\}_{i=1}^N$ be the linear co-ordinates on 
$H_{\rm orb}^*(\cX)$ dual to $\{\phi_i\}_{i=1}^N$.  
Let $\pi \colon U\times \C \to U$ be the projection, 
where $U\subset H^*_{\rm orb}(\cX)$ is 
the open subset in Assumption \ref{assump:convergence}. 
\begin{definition} 
A \seminf VHS $\tcF$ over $U$ is defined to 
be the $\pi_* \cO_{U\times \C}$-module:  
\begin{align*} 
\tcF:= H^*_{\rm orb}(\cX) \otimes \pi_*\cO_{U\times \C} 
\end{align*}
endowed with the flat connection $\nabla$ (Dubrovin connection)  
and a pairing $(\cdot,\cdot)_{\tcF}$  
\begin{equation}
\label{eq:Dubrovinconn}
\nabla := d + \frac{1}{z}\sum_{i=1}^N 
(\phi_i\circ_\tau )dt^i, 
\quad (f,g)_{\tcF}:= (f(-z), g(z))_{\rm orb}. 
\end{equation} 
It is graded by the grading operator 
$\Grading$ and the Euler vector field $E$:  
\[
\Grading := 2 z\partial_z + 2 E + 2 (\mu + \frac{n}{2}), \quad 
E := \sum_{i=1}^N (1- \frac{1}{2}\deg \phi_i) t^i \parfrac{}{t^i}  
+ \sum_{i=1}^N r^i \parfrac{}{t^i}, 
\]
where $n=\dim_\C\cX$, $c_1(T\cX)= \sum_{i} r^i \phi_i\in H^2(\cX)$ 
and $\mu\in \End(H_{\rm orb}^*(\cX))$ is defined by 
\begin{equation}
\label{eq:def_mu} 
\mu(\phi_i): = \left(\frac{\deg \phi_i}{2} - \frac{n}{2}
\right) \phi_i. 
\end{equation} 
The \seminf VHS $\tcF$ is referred to as 
\emph{quantum $D$-module} in the literature 
\cite{givental-ICM, givental-mirrorthm-toric, 
guest-qc_int, iritani-Int}. 
The standard argument (as in \cite{cox-katz,manin}) 
and the WDVV equation in orbifold Gromov-Witten theory 
\cite{AGV} show that the Dubrovin connection is flat and that 
the above data satisfy the axioms of a graded \seminf VHS. \qed 
\end{definition} 

Let $H^2(\cX,\Z)$ denote the cohomology of the constant sheaf $\Z$ 
on the topological \emph{stack} $\cX$ 
(not on the topological \emph{space}). 
This group is the set of isomorphism classes of 
topological orbifold line bundles on $\cX$. 
Let $L_\xi \to \cX$ be the orbifold line bundle corresponding to 
$\xi\in H^2(\cX,\Z)$. 
Let $0\le f_v(\xi)<1$ be the rational number such that 
the stabilizer of $\cX_v$ ($v\in \sfT$) acts on $L_\xi|_{\cX_v}$ by 
a complex number $\exp(2\pi \iu f_v(\xi))$. 
This number $f_v(\xi)$ is called the \emph{age}  
of $L_\xi$ along $\cX_v$. 
Define $G(\xi)\colon 
H^*_{\rm orb}(\cX)\to H^*_{\rm orb}(\cX)$ 
and its derivative $dG(\xi)$ by 
\begin{align*}
G(\xi)(\tau_0 \oplus \bigoplus_{v \in \sfT'} \tau_v) 
&=(\tau_0- 2\pi \iu \xi_0) \oplus \bigoplus_{v\in \sfT'} 
e^{2\pi \iu f_v(\xi)} \tau_v, \\  
dG(\xi)(\tau_0 \oplus \bigoplus_{v \in \sfT'} \tau_v) 
&=\tau_0 \oplus \bigoplus_{v\in \sfT'} e^{2\pi \iu f_v(\xi)} \tau_v, 
\end{align*} 
where $\tau_v \in H^*(\cX_v)$ and $\xi_0$ is the image of 
$\xi$ in $H^2(\cX,\Q)$. 

\begin{proposition}[{
\cite[Proposition 2.3]{iritani-Int}}]
\label{prop:Galois} 
For $\xi\in H^2(\cX,\Z)$. the map  
\[ 
dG(\xi) \colon \tcF \to G(\xi)^* \tcF, \quad 
\tcF_\tau \ni s(z) \longmapsto dG(\xi)s(z) 
\in \tcF_{G(\xi)\tau}   
\]
is a homomorphism of graded \seminf VHS's. 
We call this the \emph{Galois action} of $H^2(\cX,\Z)$ on $\tcF$.  
\end{proposition} 

We can assume that $U$ is invariant under the Galois action. 

\begin{definition} 
The \emph{quantum cohomology 
\seminf VHS $\cF$} over $U/H^2(\cX,\Z)$  
is the quotient of $\tcF \to U$ by 
the Galois action by $H^2(\cX,\Z)$ 
in Proposition \ref{prop:Galois}:  
\[
\cF := (\tcF \to U)/H^2(\cX,\Z).  
\]
The flat connection, the pairing and 
the grading operator on $\tcF$ 
induce those on $\cF$. \qed 
\end{definition}

\subsection{The fundamental solution $L(\tau,z)$}
As in Section \ref{subsec:realintstr_VHS_def}, 
the graded \seminf VHS $\tcF$ is rephrased 
as a flat connection $\hatnabla$ on the locally free sheaf 
$\cRz = H^*_{\rm orb}(\cX) \otimes \cO_{U\times \C^*}$.  
Then $\hatnabla$ defines a $\C$-local system 
$R=\Ker(\hatnabla)$ over $U\times \C^*$. 
A section of the local system $R$ 
is a cohomology-valued function $s(\tau,z)$ 
satisfying the differential equations: 
\begin{align}
\label{eq:qde_1}
&\nabla_k s = \hatnabla_k s=  
\parfrac{s}{t^k}  + \frac{1}{z} \phi_k \circ_\tau s =0, 
\quad k=1,\dots, N, \\ 
\label{eq:qde_2} 
&\hatnabla_{z\partial_z} s= 
z\parfrac{s}{z} - \frac{1}{z} E\circ_\tau s + \mu s = 0. 
\end{align} 
These equations are called \emph{quantum differential equations}. 
We give a fundamental solution $L(\tau,z)$ to the differential 
equations (\ref{eq:qde_1})  
using gravitational descendants. 
Let $\pr\colon I\cX \to \cX$ be the natural projection.  
We define the action of a class $\tau_0 \in H^*(\cX)$ on 
$H_{\rm orb}^*(\cX)$ by 
\[
\tau_0 \cdot \alpha = \pr^*(\tau_0) \cup \alpha, \quad 
\alpha \in H_{\rm orb}^*(\cX),  
\]
where the right-hand side is the cup product on $I\cX$. 
Let $\{\phi_i\}_{i=1}^N, \{\phi^i\}_{i=1}^N$ 
be mutually dual bases with respect to the 
orbifold Poincar\'{e} pairing, \emph{i.e.}  
$(\phi_i,\phi^j)_{\rm orb} = \delta_{ij}$. 
We define an $\End(H^*_{\rm orb}(\cX))$-valued 
function $L(\tau,z)$ by 
\begin{equation}
\label{eq:fundamentalsol_L}
L(\tau,z) \phi_i := e^{-\tau_{0,2}/z} \phi_i + 
\sum_{\substack{(d,l)\neq (0,0) \\ d\in \Eff_\cX}}
\sum_{k=1}^N  
\frac{\phi^k}{l!} \corr{\phi_k, \tau',\dots,\tau', 
\frac{e^{-\tau_{0,2}/z} \phi_i}{-z-\psi}}_{0,l+2,d}^\cX 
e^{\pair{\tau_{0,2}}{d}},  
\end{equation} 
where we used the decomposition (\ref{eq:tau_decomp}) 
and $1/(-z-\psi)$ in the correlator should be expanded in the series 
$\sum_{k=0}^\infty (-z)^{-k-1}\psi^k$.  

\begin{proposition}[{
\cite[Proposition 2.4]{iritani-Int}}] 
\label{prop:fundamentalsol_A} 
$L(\tau,z)$ satisfies the following differential equations: 
\begin{align}
\label{eq:diffeq_L}
\begin{split}
\nabla_k L(\tau,z) \phi_i &=0, \quad k=1,\dots,N, \\ 
\hatnabla_{z\partial_z} L(\tau,z) \phi_i &= 
L(\tau,z) (\mu \phi_i -\frac{\rho}{z}\phi_i),  
\end{split} 
\end{align}
where $\rho := c_1(T\cX)\in H^2(\cX)$. 
The $\nabla$-flat section $L(\tau,z) \phi_i$ 
is characterized by the asymptotic initial condition 
\[
L(\tau,z) \phi_i \sim e^{-\tau_{0,2}/z} \phi_i 
\]
near the large radius limit (\ref{eq:largeradiuslimit}) 
with $\tau'=0$. 
Set 
\[
z^{-\mu} z^\rho := \exp(-\mu \log z) \exp(\rho \log z). 
\]
Then we have  
\begin{gather}
\label{eq:diffeq_z_L}  
\nabla_k (L(\tau,z) z^{-\mu}z^\rho \phi_i) = 0, \quad 
\hatnabla_{z\partial_z} ( L(\tau,z) z^{-\mu} z^\rho \phi_i)=0, \\ 
\label{eq:unitarity_L}  
(L(\tau,-z) \phi_i, L(\tau,z)\phi_j)_{\rm orb} 
= (\phi_i,\phi_j)_{\rm orb}, \\ 
\label{eq:Galois_L} 
dG(\xi) L(G(\xi)^{-1} \tau, z) \alpha  
= L(\tau,z) e^{-2\pi\iu \xi_0/z} e^{2\pi\iu f_v(\xi)} 
\alpha, \quad \alpha \in H^*(\cX_v), 
\end{gather}
where $(dG(\xi), G(\xi))$ is the Galois action
associated to $\xi\in H^2(\cX,\Z)$.    
(See Section \ref{subsec:AmodelVHS}.) 
\end{proposition} 

The fundamental solution $L(\tau,z)$ is 
a priori formal power series. 
Under Assumption \ref{assump:convergence}, however, 
the convergence of $L(\tau,z)$ follows from 
the fact that it is a solution to the 
analytic differential equations. 

\subsection{The space of (multi-valued) flat sections} 
Here we apply the abstract constructions 
in Section \ref{subsec:semi-inf_period} 
to the case of the quantum cohomology \seminf VHS. 
Using the fundamental solution above, 
we will identify the spaces $\cH$ and $\cV$ 
with the Givental's loop space $\cH^\cX$ \cite{coates-givental} 
and the cohomology group $\cV^\cX$
\[
\cH^\cX:=H_{\rm orb}^*(\cX)\otimes \cO(\C^*), 
\quad 
\cV^\cX := H_{\rm orb}^*(\cX). 
\] 
For the quantum cohomology \seminf VHS, 
$\cH$ (resp.\ $\cV$) consists of 
cohomology-valued functions $s(\tau,z)$ 
satisfying (\ref{eq:qde_1}) (resp.\ both 
(\ref{eq:qde_1}) and (\ref{eq:qde_2})), 
so we can identify it with $\cH^\cX$ (resp. $\cV^\cX$) 
via $L(\tau,z)$ 
(by (\ref{eq:diffeq_L}), (\ref{eq:diffeq_z_L})):   
\begin{alignat*}{3} 
&\cH^\cX \cong \cH,& \quad 
&\alpha(z) \longmapsto L(\tau,z)\alpha(z), \\ 
&\cV^\cX \cong \cV,& \quad 
&\alpha \longmapsto  L(\tau,z) z^{-\mu}z^\rho \alpha. 
\end{alignat*} 
These identifications are understood 
throughout the paper\footnote
{However, elements of $\cH^\cX$ (or $\cV^\cX$) 
themselves are loops in the cohomology group 
(or cohomology classes) 
and are not treated as flat sections. 
When we refer to the corresponding 
sections, we explicitly denote them by  
$L(\tau,z)\alpha(z)$ (or $ L(\tau,z)z^{-\mu}z^\rho\alpha$) 
for $\alpha(z) \in \cH^\cX$ (or $\alpha \in \cV^\cX$).  
}. 
The flat connection $\hatnabla$ and the pairing 
$(\cdot,\cdot)_{\tcF}$ of the quantum cohomology 
\seminf VHS induces the operator (by (\ref{eq:diffeq_L})) 
\begin{equation} 
\label{eq:conn_z-direction}
\hatnabla_{z\partial_z} \colon \cH^\cX \to \cH^\cX, \quad 
\hatnabla_{z\partial_z} = z \parfrac{}{z} + \mu - \frac{\rho}{z} 
\end{equation} 
and the pairing (by (\ref{eq:unitarity_L})) 
\begin{equation} 
\label{eq:pairing_H} 
(\cdot,\cdot)_{\cH^\cX} \colon \cH^\cX \times \cH^\cX 
\to \cO(\C^*), \quad 
(\alpha,\beta)_{\cH^\cX} = (\alpha(-z), \beta(z))_{\rm orb}.  
\end{equation} 
As in Section \ref{subsec:semi-inf_period}, 
we can regard $\cH^\cX$ as the flat vector bundle 
$(\sfH^\cX,\hatnabla_{z\partial_z})$: 
\[
\sfH^\cX := H^*_{\rm orb}(\cX) \times \C^* \to \C^*, 
\quad 
\hatnabla_{z\partial_z} = z \parfrac{}{z} + \mu - \frac{\rho}{z}. 
\]
Then $\cV^\cX$ can be identified with the space of 
multi-valued flat sections of $\sfH^\cX$: 
\begin{equation} 
\label{eq:solutionmap_z} 
z^{-\mu}z^\rho \colon \cV^\cX \to 
\Gamma(\widetilde{\C^*}, \cO(\sfH^\cX)), 
\quad \alpha  
\mapsto z^{-\mu} z^\rho \alpha. 
\end{equation} 
The pairing $(\cdot,\cdot)_{\cV^\cX}$ on $\cV^\cX$ 
(see (\ref{eq:pairing_def_V}) for the pairing on $\cV$) 
can be written as 
\begin{align}
\label{eq:pairing_Amodel_V}
(\alpha,\beta)_{\cV^\cX}  
= (e^{\pi \iu \rho}\alpha, e^{\pi \iu \mu} \beta)_{\rm orb}.  
\end{align}
The embedding $\J_\tau\colon \tcF_\tau \hookrightarrow \cH^\cX$ 
of a fiber $\tcF_\tau$ (see (\ref{eq:embedding_J})) 
is given by the inverse of $L(\tau,z)$: 
\begin{align}
\label{eq:fundamentalsol_A}
\begin{split}
\J_\tau \alpha &= L(\tau,z)^{-1}\alpha   
 =L(\tau,-z)^\dagger \alpha \\ 
&= e^{\tau_{0,2}/z}\Biggl(\alpha + 
\sum_{\substack{(d,l)\neq (0,0) \\ d\in \Eff_\cX}} \sum_{i=1}^N 
\frac{1}{l!} 
\corr{\alpha,\tau',\dots,\tau', 
\frac{\phi_i }{z-\psi}}_{0,l+2,d}^\cX 
e^{\pair{\tau_{0,2}}{d}} \phi^i\Biggr),   
\end{split}
\end{align} 
where $L(\tau,-z)^\dagger$ is the adjoint of 
$L(\tau,-z)$ with respect to $(\cdot,\cdot)_{\rm orb}$. 
The second line follows from (\ref{eq:fundamentalsol_L}) 
and an easy computation of the adjoint $L(\tau,-z)^\dagger$.  
The image $\J_\tau \unit$ of the unit section $\unit$ 
is called the \emph{$J$-function}. 
The image $\F_\tau = \J_\tau(H^*_{\rm orb}(\cX)\otimes \cO(\C))$ 
of the embedding gives a moving subspace realization of 
the quantum cohomology \seminf VHS.   

The Galois action on $\tcF$ acts on 
$\nabla$-flat sections as $s(\tau,z) 
\mapsto dG(\xi) s(G(\xi)^{-1}\tau,z)$. 
The following lemma follows from (\ref{eq:Galois_L}). 
\begin{lemma} 
The Galois actions on $\cH^\cX$ and $\cV^\cX$ 
are given by the maps: 
\begin{align}
\label{eq:Galois_H}
G^{\cH}(\xi)(\tau_0 \oplus \bigoplus_{v\in \sfT'} \tau_v) 
& = e^{-2\pi \iu \xi_0/z}\tau_0 \oplus \bigoplus_{v\in \sfT'} 
e^{-2\pi \iu \xi_0/z} e^{2\pi \iu f_v(\xi)} \tau_v,  \\ 
\label{eq:Galois_V}
G^{\cV}(\xi)(\tau_0 \oplus \bigoplus_{v\in \sfT'} \tau_v)  
& = e^{-2\pi \iu \xi_0}\tau_0 \oplus \bigoplus_{v\in \sfT'} 
e^{-2\pi \iu \xi_0} e^{2\pi \iu f_v(\xi)} \tau_v,  
\end{align}
where we used the decomposition 
$\cH^{\cX} = \bigoplus_{v\in \sfT} 
H^*( {\cX_v})\otimes \cO(\C^*)$. 
\end{lemma} 

The Galois actions on $\cH^\cX$, $\cV^\cX$ 
can be viewed as the monodromy of 
$\nabla$ over $U/H^2(\cX,\Z)$. 
The monodromy transformation of 
$\hatnabla_{z\partial_z}$ on $\C^*$ 
is given by 
\begin{equation}
\label{eq:z-monodromy_V}
e^{-2\pi \iu \mu} e^{2\pi \iu \rho} 
\colon \cV^\cX \longrightarrow \cV^\cX. 
\end{equation} 
This coincides with the Galois action $(-1)^n G^\cV([K_\cX])$. 
Here, $[K_\cX]$ is the class of the canonical line bundle.
When $\cX$ is Calabi-Yau, 
\emph{i.e.} $K_\cX$ is trivial,  
the pairing $(\cdot,\cdot)_{\cV^\cX}$ 
is either symmetric or anti-symmetric depending on
whether $n$ is even or odd.   
In general, this pairing is neither symmetric nor anti-symmetric.  

Recall that a real structure on the \seminf VHS 
$\cF$ is given by a sub $\R$-local system $R_\R$ 
of the $\C$-local system $R$ defined by $\hatnabla$ 
(Definition \ref{def:realintstr}).  
Therefore, a real structure on the 
quantum cohomology \seminf VHS $\cF$ is identified 
with a monodromy-invariant real subspace 
$\cV^\cX_\R$ in the space $\cV^\cX$ of 
multi-valued $\hatnabla$-flat sections.

\begin{proposition}
\label{prop:char_A_real_int_str}
A real (integral) structure $R_\R$ on the 
quantum cohomology \seminf VHS $\cF$   
is given by a real subspace $\cV_\R^{\cX}$ 
(resp. integral lattice $\cV_\Z^\cX$) 
of $\cV^\cX = H^*_{\rm orb}(\cX)$ satisfying 
\begin{itemize}
\item[(i)] $\cV^{\cX} = \cV_\R^{\cX} \otimes_\R \C$ 
(resp. $\cV^{\cX} = \cV^\cX_\Z \otimes_\Z \C$);  
\item[(ii)] $\cV_\R^\cX$ (resp. $\cV_\Z^\cX$) is invariant 
under the Galois action (\ref{eq:Galois_V}); 
\item[(iii)] The pairing (\ref{eq:pairing_Amodel_V}) 
restricted on $\cV^\cX_\R$ (resp. $\cV_\Z^\cX$) takes values in $\R$
(resp. takes values in $\Z$ and is unimodular).    
\end{itemize} 
\end{proposition} 

A real structure $R_\R$ on the quantum cohomology \seminf VHS 
induces the real subspace $\cH^\cX_\R$ of $\cH^\cX$ 
(see (\ref{eq:real_subspaces})):  
\begin{align*} 
\cH^\cX_\R := & \left\{\alpha(z) \in \cH^\cX \;;\; 
\begin{array}{l} 
\text{$L(\tau,\alpha) \alpha(z)$ 
belongs to the fiber of the }   \\ 
\text{$\R$-local system $R_\R$ 
at each $(\tau,z) \in U \times S^1$.}
\end{array} 
\right \}.   
\end{align*}
Note that for $\alpha(z) \in \cH^\cX$, 
$L(\tau,z)\alpha(z)$ is not necessarily 
a section of $R_\R|_{U\times S^1}$ 
since it may not be flat. 
($\hatnabla_{z\partial_z}$ may have monodromy.) 
Let $\kappa_\cH$ and $\kappa_\cV$ denote 
the real involutions of $\cH^\cX$ and $\cV^\cX$  
introduced in Section \ref{subsec:semi-inf_period}.   
We decompose the Galois action on $\cH^\cX$ as 
\[
G^{\cH}(\xi) = e^{-2\pi \iu \xi_0/z} G_0^{\cH}(\xi), \quad 
G_0^{\cH}(\xi) := \bigoplus_{v\in \sfT} e^{2\pi \iu f_v(\xi)}. 
\]
\begin{proposition}
\label{prop:basic_realstr} 
For any real structure on the 
quantum cohomology \seminf VHS $\cF$, 
the following holds. 
For a real class $\tau_{0,2}\in H^2(\cX,\R)$, we have 
\begin{gather}
\label{eq:H2_imaginary}
\kappa_\cH (\tau_{0,2}/z) + (\tau_{0,2}/z) \kappa_{\cH} =0, \quad 
\kappa_\cV \tau_{0,2} + \tau_{0,2} \kappa_\cV =0, \\  
\label{eq:G_0_real} 
G_0^{\cH}(\xi) \kappa_\cH = \kappa_\cH G_0^\cH(\xi), \\ 
(z\partial_z + \mu) \kappa_\cH + 
\label{eq:grading_imaginary_H} 
\kappa_{\cH} (z\partial_z +\mu) =0, \\
\label{eq:relation_kappa_H_V} 
\kappa_{\cH} = z^{-\mu} \kappa_{\cV} z^\mu, \quad 
\end{gather}  
where the last equality holds when we regard an element 
of $\cH^\cX$ as a $\cV^\cX$-valued function over $S^1=\{|z|=1\}$.   
Moreover, if $\cX$ satisfies the following condition: 
\begin{equation}
\label{eq:separation}
f_v(\xi) = f_{v'}(\xi), \ \forall \xi\in H^2(\cX,\Z) 
\Longrightarrow v=v',  
\end{equation} 
then we have 
\begin{equation}
\label{eq:kappa_induces_Inv}
\begin{split}
\kappa_\cH(H^*(\cX_v)\otimes \cO(\C^*)) 
&= H^*(\cX_{\inv(v)})\otimes \cO(\C^*), \\ 
\kappa_\cV(H^*(\cX_v)) & = H^*(\cX_{\inv(v)}).   
\end{split} 
\end{equation}
When (\ref{eq:kappa_induces_Inv}) holds, $\kappa_\cV$ satisfies
\begin{equation}
\label{eq:kappaV_weightfiltr}
\kappa_\cV(\alpha) \in  \cC(\alpha) + H^{>2k}(\cX_{\inv(v)}), \quad 
\alpha\in H^{2k}(\cX_v)
\end{equation} 
for some complex antilinear isomorphism 
$\cC:H^{2k}(\cX_v) \to H^{2k}(\cX_{\inv(v)})$. 
\end{proposition} 
\begin{proof} 
Because all the $f_v(\xi)$'s are rational numbers, 
we can find an integer $m>0$ such that 
$(G_0^\cH(\xi))^m=\id$. 
Then $(G^\cH(\xi))^m= e^{-2\pi\iu m \xi_0/z}$. 
Because the Galois action preserves the real structure, 
$\xi_0/z$ has to be purely imaginary on $\cH^\cX$. 
Hence $\tau_{0,2}/z$ is purely imaginary on $\cH^\cX$ for 
every $\tau_{0,2}\in H^2(\cX,\R)$.  
From this and (\ref{eq:solutionmap_z}), we find that 
the multiplication by $\tau_{0,2}$ is purely imaginary on $\cV^\cX$. 
Thus we have (\ref{eq:H2_imaginary}). 
From $G_0^{\cH}(\xi) = e^{2\pi\iu \xi_0/z}G^{\cH}(\xi)$, 
we have (\ref{eq:G_0_real}).  
Because $\hatnabla_{z\partial_z}=z\partial_z + \mu-\rho/z$ 
is purely imaginary on $\cH^\cX$ (\ref{eq:property_kappaH}) 
and so is $\rho/z$, 
we have (\ref{eq:grading_imaginary_H}).  
By (\ref{eq:solutionmap_z}), 
$\kappa_\cH$ and $\kappa_\cV$ are related by 
\[
\kappa_\cH = z^{-\mu} z^\rho \kappa_{\cV} z^{-\rho} z^{\mu}, 
\]
where $z$ is assumed to be in $S^1$ 
and both hand sides act on $\cV^\cX$-valued 
functions over $S^1$. 
Since $z^\rho = \exp(\rho \log z)$ 
is real on $\cV^\cX$ when $z\in S^1$, 
we have (\ref{eq:relation_kappa_H_V}). 
Under the condition (\ref{eq:separation}), 
the decomposition $\cH^\cX = \bigoplus_{v\in \sfT} 
H^*(\cX_v)\otimes \cO(\C^*)$ 
is the simultaneous eigenspace decomposition for $G_0^\cH(\xi)$, 
$\xi\in H_2(\cX,\Z)$.  
Therefore, (\ref{eq:kappa_induces_Inv}) follows 
from $\ov{e^{2\pi\iu f_v(\xi)}} = e^{2\pi\iu f_{\inv(v)}(\xi)}$ 
and the reality of $G^\cH_0(\xi)$. 
Let $\omega$ be a K\"{a}hler class on $\cX$. 
The action of $\omega$ on $H^*(\cX_v)$ is nilpotent. 
In general, a nilpotent operator $\omega$ on a vector space  
defines an increasing filtration $\{W_k\}_{k\in \Z}$ on it, 
called a {\it weight filtration}, 
which is uniquely determined by the conditions: 
\[
\omega W_k \subset W_{k-2}, \quad 
\omega^k \colon \Gr^W_{k} \cong \Gr^W_{-k} 
\]
where $\Gr^W_k = W_k/W_{k-1}$. 
By the Lefschetz decomposition, we know that 
$W_k = H^{\ge n_v-k}(\cX_v)$ in this case 
($n_v:=\dim_\C\cX_v$).  
Since $\kappa_\cV$ anti-commutes with $\omega$ 
by (\ref{eq:H2_imaginary}), $\kappa_\cV$ preserves this filtration. 
This shows (\ref{eq:kappaV_weightfiltr}). 
Here, $\cC$ is the isomorphism on the associated graded quotient 
induced from $\kappa_\cV$. 
\end{proof}

\subsection{Purity and polarization}
\label{subsec:pure_polarized} 
For an arbitrary real structure, 
we study a behavior of the quantum cohomology \seminf VHS $\cF$ near 
the large radius limit point (\ref{eq:largeradiuslimit}). 
We show that it is pure and polarized 
(in the sense of Definitions 
\ref{def:pure}, \ref{def:polarized}) 
under suitable conditions. 
Recall that when $\tcF\to U$ is pure,  
this defines a Cecotti-Vafa structure 
on the vector bundle $K\to U$ by Proposition \ref{prop:CV-str}. 

\begin{theorem}
\label{thm:pure_polarized} 
Let $\cX$ be a smooth Deligne-Mumford stack with 
a projective coarse moduli space. 
Let $\cF$ be the quantum cohomology \seminf VHS 
of $\cX$ and take a real structure on $\cF$. 
Let $\omega$ be a K\"{a}hler class on $\cX$. 

{\rm (i)} 
Assume that the real structure satisfies (\ref{eq:kappa_induces_Inv}). 
Then $\cF$ is pure at $\tau=-x\omega$ when 
the real part $\Re(x)$ is sufficiently big. 

{\rm (ii)}  
Assume moreover that the real structure satisfies 
(cf. (\ref{eq:kappaV_weightfiltr}))
\begin{align}
\label{eq:leadingterm_kappaV}
\begin{split}
&\kappa_\cV(\alpha) \in (-1)^k \R_{>0} \inv^*(\ov{\alpha}) 
+ H^{>2k}(\cX_{\inv(v)}), \\ 
\text{or equivalently} \quad &\kappa_\cH(\alpha) 
= (-1)^k\R_{>0} \inv^*(\ov{\alpha}) z^{-2k+n_v} 
+ O(z^{-2k+n_v-1})    
\end{split} 
\end{align} 
for $\alpha\in H^{2k}(\cX_v) \subset H_{\rm orb}^*(\cX)$, 
$n_v=\dim_\C\cX_v$. 
Then the Hermitian metric 
$h(\cdot,\cdot)=g(\kappa(\cdot),\cdot)$ 
on the vector bundle $K\to U$ satisfies 
\[
(-1)^{\frac{p-q}{2}} h(u,u) >0, 
\quad u\in H^{p,q}(\cX_v) \subset K_{-x\omega}, \quad u\neq 0 
\]
for sufficiently big $\Re(x)>0$, 
where we identify $K_\tau$ with 
$\tcF_\tau/z\tcF_\tau \cong H^*_{\rm orb}(\cX)$.  
In particular, if $H^*_{\rm orb}(\cX)$ 
consists only of the $(p,p)$ part, i.e. 
$H^{2p}(\cX_v) = H^{p,p}(\cX_v)$ for all $v\in \sfT$ and $p\ge 0$, 
then $\cF$ is polarized at $\tau=-x\omega$ 
for sufficiently big $\Re(x)>0$. 
\end{theorem} 

\begin{remark}
\label{rem:after_pure_polarized} 
(i) The condition (\ref{eq:kappa_induces_Inv}) is satisfied 
when $\cX$ has enough line bundles to separate the inertia components 
(see (\ref{eq:separation}) in Proposition \ref{prop:basic_realstr}). 
In particular, (\ref{eq:kappa_induces_Inv}) is always satisfied when 
$\cX$ is a manifold.  

(ii) 
We can consider the 
\emph{algebraic quantum cohomology \seminf VHS}.  
Let $A^*(\cX)_\C$ denote the Chow ring of $\cX$ over $\C$. 
We set $\HH^*(\cX_v) := \Image(A^*(\cX_v)_\C \to H^*(\cX_v))$ and 
define $\HH_{\rm orb}^*(\cX) := \bigoplus_{v\in \sfT} \HH^*(\cX_v)$. 
The algebraic quantum cohomology \seminf VHS is defined to be 
\[
\HH^*_{\rm orb}(\cX) \otimes 
\pi_* \cO_{(U\cap \HH_{\rm orb}(\cX))\times \C} 
\]
with the restriction of Dubrovin connection, the grading operator 
and pairing,   
modulo the Galois action given by an element of $\Pic(\cX)$. 
Here we used the fact that the quantum product 
among classes in $\HH_{\rm orb}^*(\cX)$ 
again belongs to $\HH_{\rm orb}^*(\cX)$.  
This follows from the algebraic construction 
of orbifold Gromov-Witten theory \cite{AGV}.  
When we assume the Hodge conjecture for all $\cX_v$, 
each $\HH^*(\cX_v)$ has the Poincar\'{e} duality and
the orbifold Poincar\'{e} pairing is non-degenerate 
on $\HH_{\rm orb}^*(\cX)$. 
Under this assumption, 
\emph{the algebraic quantum cohomology \seminf VHS 
is pure and polarized at $\tau=-x\omega$ for  
a K\"{a}hler class $\omega\in \HH^2(\cX)$ and $\Re(x)\gg 0$}  
if the conditions corresponding to (\ref{eq:kappa_induces_Inv}) and 
(\ref{eq:leadingterm_kappaV}) are satisfied. 
The proof below applies to the algebraic quantum cohomology 
\seminf VHS without change. 
Note that the Poincar\'{e} duality of $\HH^*(\cX_v)$ also 
implies the Hard Lefschetz of it
used in the proof below.  
\end{remark} 

\begin{remark} 
Hertling \cite{hertling-tt*} 
and Hertling-Sevenheck \cite{hertling-sevenheck} 
studied similar problems for general TERP structures. 
They considered the change of TERP structures induced 
by the rescaling $z \mapsto r z$ of the parameter $z$. 
This rescaling with $r \to \infty$ 
is called Sabbah orbit in \cite{hertling-sevenheck} 
and is equivalent to the flow of minus the Euler vector field: 
$\tau \mapsto \tau - \rho \log r$ for $\tau \in H^2(\cX)$. 
When $\cX$ is Fano and $\omega=c_1(\cX)=\rho$, 
the large radius limit corresponds to 
the Sabbah orbit\footnote{
The author thanks 
Claus Hertling for this remark.}, 
and the conclusions 
in Theorem \ref{thm:pure_polarized} can be deduced from 
\cite[Theorem 7.3]{hertling-sevenheck} in this case. 
\end{remark} 

\begin{remark}
\label{rem:sabbah} 
Singularity theory gives a \seminf VHS with a real structure. 
According to the recent work of Sabbah \cite[Section 4]{sabbah}, 
the \seminf VHS arising from a cohomologically tame function 
on an affine manifold is pure and polarized.  
In Section \ref{sec:exampleP1tt*}, we use this result 
to see that the $tt^*$-geometry of $\Proj^1$ 
is pure and polarized everywhere. 
\end{remark} 

The rest of this section is 
devoted to the proof of Theorem \ref{thm:pure_polarized}. 

From Equation (\ref{eq:fundamentalsol_A}), 
we see that 
$e^{x\omega/z} \J_{-x\omega}(\varphi) \to  \varphi$ 
as $\Re(x)\to \infty$. 
Thus, in the moving subspace realization, 
the Hodge structure 
$\F_{-x\omega} = \J_{-x\omega}(\cF_{-x\omega})$ 
has the asymptotics: 
\[
\F_{-x \omega} \sim e^{-x\omega/z}\F_{\rm lim} \quad 
\text{as $\Re(x)\to \infty$}, 
\]
where $\Flim:= H^*_{\rm orb}(\cX)\otimes \cO(\C)$ 
is the limiting Hodge structure. 
This is an analogue of the {\it nilpotent orbit theorem} 
\cite{schmid}  
in quantum cohomology. 
First we study the behavior of the nilpotent orbit 
$x\mapsto e^{-x\omega/z} \Flim$ for $\Re(x)\gg 0$ 
(see Proposition \ref{prop:nilpotentorb_pure} below). 

{\bf (Step 1)} 
We study the purity of the \seminf VHS 
$x\mapsto e^{-x\omega/z} \Flim$, 
\emph{i.e.} if the natural map 
\begin{equation}
\label{eq:nilpotentorb_pure}
e^{-x\omega/z} \Flim \cap \kappa_{\cH}(e^{-x\omega/z}\Flim) 
\longrightarrow e^{-x\omega/z} (\Flim/z\Flim) 
\cong e^{-x\omega/z}H^*_{\rm orb}(\cX)  
\end{equation} 
is an isomorphism (see (\ref{eq:FcapkappaF}) 
in Proposition \ref{prop:purity_condition}). 
Under the condition (\ref{eq:kappa_induces_Inv}), 
this is equivalent to that the map   
\[
e^{-x\omega/z} H^*( \cX_v)\{z\} \cap 
\kappa_\cH (e^{-x\omega/z} H^*( \cX_{\inv(v)} )\{z\}) 
\to   e^{-x\omega/z} H^*( \cX_v)   
\]
is an isomorphism for each $v\in \sfT$. 
Here we put $H^*(\cX_v)\{z\} := H^*(\cX_v) \otimes \cO(\C)$. 
Since $\kappa_\cH e^{-x\omega/z} = e^{\ov{x} \omega/z} \kappa_\cH$ 
(see (\ref{eq:H2_imaginary})), this is equivalent to that 
\[
H^*( \cX_v)\{z\} \cap 
e^{2t\omega/z} 
\kappa_\cH(H^*( \cX_{\inv(v)})\{z\}) \to H^*( \cX_{v}), 
\quad t:=\Re(x)
\]
is an isomorphism. 
We further decompose this into $(z\partial_z + \mu)$-eigenspaces. 
Because $z\partial_z +\mu$ 
is purely imaginary (\ref{eq:grading_imaginary_H}), 
the above map between the $(z\partial_z+\mu)$-eigenspaces  
of the eigenvalue 
$\frac{1}{2}(-k+\age(v)-\age(\inv(v)) )$ is of the form: 
\[
\left
(\bigoplus_{l\ge 0} H^{n_v-k-2l}( \cX_v) z^l
\right) 
\cap 
e^{2t\omega/z}\kappa_{\cH}
\left(\bigoplus_{l\ge 0} H^{n_v+k-2l}( \cX_{\inv(v)}) z^l
\right) 
\to  H^{n_v-k}(\cX_v).  
\]
Here, $n_v=\dim_\C \cX_v$ and $k$ is an integer such that 
$n_v-k$ is even. 
By using (\ref{eq:relation_kappa_H_V}), 
we find that this map is conjugate 
(via $z^{\mu+(k-\iota_v+\iota_{\inv(v)})/2}$) 
to the following map: 
\begin{equation}
\label{eq:kappaV_transversality}
H^{\le n_v-k}(\cX_v) \cap e^{2 t\omega} 
\kappa_\cV(H^{\le n_v+k}(\cX_{\inv(v)})) \to H^{n_v-k}(\cX_v)   
\end{equation} 
which is induced by 
$H^{\le n_v-k}(\cX_v) \to 
H^{\le n_v-k}(\cX_v)/H^{\le n_v-k-2}(\cX_v) 
\cong H^{n_v-k}(\cX_v)$. 
We will show that this becomes an isomorphism for 
$t=\Re(x)\gg 0$ 
in Lemma \ref{lem:nilpotentorb_pure} below. 
 
Let $\fra \colon H^*(\cX_v) \to H^{*+2}(\cX_v)$ 
be the operator defined by $\fra(\phi):=\omega\cup \phi$. 
There exists an operator 
$\fra^\dagger \colon H^*(\cX_v) \to H^{*-2}(\cX_v)$ 
such that $\fra$ and $\fra^\dagger$ generate
the Lefschetz $\mathfrak{sl}_2$-action on $H^*(\cX_v)$:  
\[
[\fra, \fra^\dagger] = h, \quad 
[h,\fra] = 2\fra, \quad [h,\fra^\dagger]
=-2\fra^\dagger,    
\]
where $h: = \deg - n_v$ is the (shifted) grading operator. 
Note that $\fra^\dagger$ is uniquely determined by the 
above commutation relation and that $\fra^\dagger$ annihilates 
the primitive cohomology 
$PH^{n_v-k}(\cX_v):= \Ker(\fra^{k+1} \colon H^{n_v-k}(\cX_v) 
\to H^{n_v+k+2}(\cX_v))$. 

\begin{lemma}
\label{lem:expa_expadagger}
The map $e^{-\fra} e^{\fra^\dagger} \colon H^*(\cX_v) \to H^*(\cX_v)$ 
sends $H^{\ge n_v-k}(\cX_v)$ onto $H^{\le n_v+k}(\cX_v)$  
isomorphically. Moreover, for 
$u \in \fra^{j} PH^{n_v-k-2j}(\cX_v)\subset H^{n_v-k}(\cX_v)$, 
one has 
\[
e^{-\fra} e^{\fra^\dagger} u = (-1)^{k+j} \frac{j!}{(k+j)!}
\omega^k u + H^{< n_v+k}(\cX_v). 
\]
\end{lemma} 
\begin{proof}
An easy calculation shows that 
\[
e^{-\fra} e^{\fra^\dagger} \fra= 
- \fra^\dagger e^{-\fra} e^{\fra^\dagger}.  
\]
Therefore, $e^{-\fra}e^{\fra^\dagger}$ should send 
the weight filtration for the nilpotent operator 
$\fra$ to that for $\fra^\dagger$.  
But the weight filtration for $\fra$ is 
$\{H^{\ge n_v-k}\}_k$ 
and that for $\fra^\dagger$ is $\{H^{\le n_v+k}\}_k$ 
(see the proof of Proposition \ref{prop:basic_realstr} 
for weight filtration).  
Take $u \in \fra^j PH^{n_v-k-2j}(\cX_v)$. 
Put $u = \fra^j \phi$ for $\phi \in PH^{n_v-k-2j}(\cX_v)$. 
We calculate 
\begin{align*}
e^{-\fra}e^{\fra^\dagger} u  &= e^{-\fra}e^{\fra^\dagger} \fra^j \phi 
= (-\fra^\dagger)^j e^{-\fra}e^{\fra^\dagger} \phi 
=  (-\fra^\dagger)^j e^{-\fra} \phi \\
&= (-\fra^\dagger)^j \frac{(-1)^{k+2j}}{(k+2j)!} \fra^{k+2j} \phi  
+ \text{ lower degree term},  
\end{align*} 
where in the second line we used that 
$e^{-\fra}e^{\fra^\dagger}u \in H^{\le n_v+k}(\cX_v)$. 
Using $\fra^\dagger \fra^l u = l(k+2j+1-l) \fra^{l-1} u$, 
we arrive at the formula for $e^{-\fra}e^{\fra^\dagger}u$.  
\end{proof}

\begin{lemma}
\label{lem:nilpotentorb_pure} 
The map (\ref{eq:kappaV_transversality}) is an isomorphism 
for sufficiently big $t>0$. 
Moreover, $u\in H^{n_v-k}(\cX_v)$ corresponds to 
an element of the form 
\[
(2t)^{(\deg +k-n_v)/2}(e^{\fra^\dagger} u + O(t^{-1}))  
\in H^{\le n_v-k}(\cX_v)  \cap e^{2\omega t}
\kappa_\cV(H^{\le n_v+k}(\cX_{\inv(v)})) 
\] 
under (\ref{eq:kappaV_transversality}), where $(2t)^{\deg/2}$ 
is defined by $(2t)^{\deg/2} = (2t)^k$ on $H^{2k}(\cX_v)$. 
\end{lemma} 
\begin{proof}
First we rescale (\ref{eq:kappaV_transversality}) by 
$(2t)^{-\deg/2}$:  
\[
\begin{CD}
H^{\le n_v-k}(\cX_v) \cap e^{2\omega t}
\kappa_\cV(H^{\le n_v+k}(\cX_{\inv(v)}))  @>>> H^{n_v-k}(\cX_v)  \\ 
@VV{(2t)^{-\deg/2}}V @VV{(2t)^{-\deg/2}}V \\ 
H^{\le n_v-k}(\cX_v) \cap e^{\omega} 
\kappa_t (H^{\le n_v+k}(\cX_{\inv(v)})) @>>> H^{n_v-k}(\cX_v), 
\end{CD} 
\]
where $\kappa_t := (2t)^{-\deg/2} \kappa_\cV (2t)^{\deg/2}$. 
Since the column arrows are isomorphisms for all $t\in \R$, 
it suffices to show that the bottom arrow is an isomorphism 
for $t\gg 0$.  
Observe that the expected dimension of 
$H^{\le n_v-k}(\cX_v) \cap e^{2\omega} 
\kappa_t (H^{\le n_v+k}(\cX_{\inv(v)}))$ equals 
$\dim H^{n_v-k}(\cX_v)$ by Poincar\'{e} duality. 
Thus that the bottom arrow becomes an isomorphism 
is an open condition for $\kappa_t$. 
By (\ref{eq:kappaV_weightfiltr}) 
in Proposition \ref{prop:basic_realstr}, 
we have 
\begin{equation}
\label{eq:kappat_goestoC}
\kappa_t = \cC + O(t^{-1}), 
\end{equation} 
for a degree preserving isomorphism 
$\cC\colon H^*(\cX_{\inv(v)})\cong H^*(\cX_v)$. 
Therefore, we only need to check that the map at $t=\infty$ 
\begin{equation}
\label{eq:transversality_tinfinity}
H^{\le n_v-k}(\cX_v) \cap e^{\fra} H^{\le n_v+k}(\cX_v) 
\to H^{n_v-k}(\cX_v) 
\end{equation}
is an isomorphism (recall that $\fra=\omega\cup$). 
Note that this factors through $\exp(-\fra^\dagger)$ as  
\[
\begin{CD}
H^{\le n_v-k} \cap e^{\fra} H^{\le n_v+k}
@>{\exp(-\fra^\dagger)}>> 
H^{\le n_v-k}\cap e^{-\fra^\dagger}e^{\fra}  
H^{\le n_v+k} 
@>>> H^{n_v-k}, 
\end{CD} 
\]
where we omitted the space $\cX_v$ from the notation.  
The second map is induced from the 
projection $H^{\le n_v-k} \to H^{n_v-k}$ 
again. 
Because $e^{-\fra^\dagger}e^{\fra}(H^{\le n_v+k})=H^{\ge n_v-k}$ 
by Lemma \ref{lem:expa_expadagger}, 
we know that the map 
(\ref{eq:transversality_tinfinity}) is an isomorphism 
and that the inverse map is given by 
$u\mapsto \exp(\fra^\dagger)u$. 
Now the conclusion follows. 
\end{proof} 

\begin{proposition}
\label{prop:nilpotentorb_pure} 
Assume that (\ref{eq:kappa_induces_Inv}) holds.  
Then the nilpotent orbit 
$x\mapsto e^{-x\omega/z} \Flim$ is pure 
for sufficiently big $t=\Re(x)>0$
i.e. the map (\ref{eq:nilpotentorb_pure}) is an isomorphism 
for $t\gg 0$. 
The inverse image of 
$e^{-x\omega/z} u$, $u\in H^{n_v-k}(\cX_v)$
under (\ref{eq:nilpotentorb_pure}) is of the form 
$e^{-x\omega/z} \varpi_t(u)$ with 
\begin{equation}
\label{eq:lift_to_FcapkappaF} 
\varpi_t(u) = z^{-\mu-(k-\iota_v+ \iota_{\inv(v)})/2} 
(2t)^{(\deg + k-n_v)/2} 
(e^{\fra^\dagger} u+ O (t^{-1})) 
\in \bigoplus_{l\ge 0} H^{n_v-k-2l}(\cX_v)z^l. 
\end{equation} 
When $u = \fra^j \phi$ and $\phi \in PH^{n_v-k-2j}(\cX_v)$, 
we have  
\begin{equation}
\label{eq:nilpotentorb_Hermitianmetric}
(\kappa_\cH(e^{-x\omega/z}\varpi_t(u)),
e^{-x\omega/z} \varpi_t(u))_{\cH^\cX} 
= \frac{(2t)^{k} j!}{(k+j)!}\int_{\cX_v} \omega^{k+2j}
\phi \cup \inv^*\cC(\phi)   + O(t^{k-1})   
\end{equation}
where $\cC\colon H^*(\cX_v) \to H^*(\cX_{\inv(v)})$ is 
the isomorphism appearing in (\ref{eq:kappaV_weightfiltr}) 
and $(\cdot,\cdot)_{\cH^\cX}$ is given in (\ref{eq:pairing_H}).  
If moreover $u\in H^{p,q}(\cX_v)\setminus \{0\}$ 
and the condition (\ref{eq:leadingterm_kappaV}) holds, 
\[
(-1)^{(p-q)/2}(\kappa_\cH(e^{-x\omega/z}\varpi_t(u)),
e^{-x\omega/z} \varpi_t(u))_{\cH^\cX} >0 
\]
for $t=\Re(x)\gg 0$.
(Here $p+q = n_v-k$.)  
\end{proposition} 
\begin{proof} 
The purity of $e^{-x\omega/z} \Flim$ 
and the formula for $\varpi_t(u)$ 
follow from Lemma \ref{lem:nilpotentorb_pure} and 
the discussion preceding (\ref{eq:kappaV_transversality}). 
Putting $c= (-k +\iota_v-\iota_{\inv(v)})/2$, we calculate 
\begin{align*}
&(\kappa_\cH(e^{-x\omega/z}\varpi_t(u)),
e^{-x\omega/z} \varpi_t(u))_{\cH^\cX} 
 = (\kappa_\cH(\varpi_t(u)), e^{-2t\omega/z} \varpi_t(u))_{\cH^\cX} \\
& = (2t)^{k-n_v} 
(z^{-\mu-c}\kappa_\cV (2t)^{\deg/2}
(e^{\fra^\dagger} u + O(t^{-1})), 
z^{-\mu+c} 
(2t)^{\deg/2}(e^{-\fra}e^{\fra^\dagger} u +O(t^{-1})))_{\cH^\cX} \\   
& = (2t)^{k} ( (-1)^{-\mu-c} \kappa_t 
( e^{\fra^\dagger} u + O(t^{-1})), 
e^{-\fra}e^{\fra^\dagger} u + O(t^{-1}))_{\rm orb}  
\end{align*} 
where we used $e^{-2t\omega/z} z^{-\mu} (2t)^{\deg/2} 
= z^{-\mu}(2t)^{\deg/2} e^{-\fra}$ and (\ref{eq:relation_kappa_H_V}) 
in the second line (we assume $|z|=1$) and set 
$\kappa_t := (2t)^{-\deg/2} \kappa_\cV (2t)^{\deg/2}$ again 
in the third line. 
From (\ref{eq:kappat_goestoC}),    
the highest order term in $t$ becomes 
\[
(2t)^{k} ( (-1)^{-\mu-c} e^{-\fra^\dagger} \cC(u), 
e^{-\fra}e^{\fra^\dagger} u)_{\rm orb}.  
\] 
Note that $\cC$ anticommutes with $\fra,\fra^\dagger$ 
by (\ref{eq:H2_imaginary}). 
By a calculation using Lemma \ref{lem:expa_expadagger}, 
we find that this equals the highest order term 
of the right-hand side of (\ref{eq:nilpotentorb_Hermitianmetric}). 
The last statement on positivity follows from the 
classical Hodge-Riemann bilinear inequality: 
\[
(-1)^{(p-q)/2} (-1)^{(n_v-k)/2 - j} 
\int_{\cX_v} \omega^{k+2j} \phi \cup \ov{\phi} >0  
\]
for $\phi \in PH^{n_v-k-2j}(\cX_v)\cap 
H^{p-j,q-j}(\cX_v)\setminus\{0\}$, $n_v-k$ even. 
\end{proof}

{\bf (Step 2)}  
Next we show that $x\mapsto \F_{-x\omega}$ is pure for 
$t=\Re(x)\gg 0$.   
We set $\F'_{-x\omega} = e^{x\omega/z}\F_{-x\omega}$. 
Again by (\ref{eq:FcapkappaF}) 
in Proposition \ref{prop:purity_condition} 
and $\kappa_\cH e^{-x\omega/z} = 
e^{\ov{x}\omega/z} \kappa_\cH$, 
it is sufficient to show that 
\[
\F'_{-x\omega} \cap e^{2 t\omega/z} \kappa_{\cH}(\F'_{-x\omega}) 
\longrightarrow \F'_{-x\omega}/ z \F'_{-x\omega} 
\]
is an isomorphism. 
Put $\kappa^t = e^{2t\omega/z}\kappa_{\cH}$ 
($\kappa^t$ is different from $\kappa_t$ 
appearing in (\ref{eq:kappat_goestoC})).  
Fix a basis $\{\phi_1,\dots,\phi_N\}$ of 
$H^*_{\rm orb}(\cX)$. 
Define an $N\times N$ matrix $A_t(z,z^{-1})$ by 
\begin{equation}
\label{eq:matrixA}
[\kappa^t(\phi_1),\dots,\kappa^t(\phi_N)] 
= [\phi_1,\dots,\phi_N] A_t(z,z^{-1}).  
\end{equation} 
This matrix $A_t$ is a Laurent polynomial in $z$ 
(by (\ref{eq:relation_kappa_H_V})) and 
a polynomial in $t$. 
We already showed that (\ref{eq:nilpotentorb_pure}) 
is an isomorphism for $t=\Re(x) \gg 0$. Therefore, 
$A_t(z)$ admits the Birkhoff factorization 
$A_t(z) = B_t(z) C_t(z)$ for $t\gg 0$, 
where $B_t\colon \D_0 \to GL_N(\C)$ with $B_t(0)=\unit$ and 
$C_t\colon \D_\infty \to GL_N(\C)$ 
(see Remark \ref{rem:Birkhoff_Iwasawa}).  
The matrix $B_t(z)$ here is given by 
\[
[\varpi_t(\phi_1),\dots, \varpi_t(\phi_N)] 
= [\phi_1,\dots,\phi_N] B_t(z)
\] 
for $\varpi_t(\phi_i)$ appearing in (\ref{eq:lift_to_FcapkappaF}).  
In particular, $B_t(z)$ and $C_t(z)$ are polynomials in $z$ 
and $z^{-1}$ respectively and have 
at most polynomial growth in $t$. 
We define 
$Q_x \colon \Proj^1\setminus \{0\} \to GL_N(\C)$ by 
\begin{equation}
\label{eq:matrixQ}
[j_1,\dots,j_N] = [\phi_1,\dots,\phi_N] Q_x(z), \quad 
j_i :=  e^{x\omega/z} \J_{-x\omega}(\phi_i)
\end{equation} 
where $\J_{\tau}$ is given in (\ref{eq:fundamentalsol_A}). 
The vectors $j_1,\dots, j_N$ form a basis of $\F'_{-t\omega}$
and $Q_x(\infty) = \unit$. 
Note that $Q_x = \unit + O(e^{-\epsilon_0 t})$ as 
$t=\Re(x) \to \infty$  
for $\epsilon_0 := \min(\pair{\omega}{d}\;;\; 
d\in \Eff_\cX\setminus\{0\})$. 
From (\ref{eq:matrixA}) and (\ref{eq:matrixQ}), we find 
\[
[\kappa^t(j_1),\dots,\kappa^t(j_N)] 
= [j_1,\dots,j_N] Q_x^{-1} A_t \ov{Q}_x,  
\]
where $\ov{Q}_x$ is the complex conjugate 
of $Q_x$ with $z$ restricted to $S^1 = \{|z|=1\}$. 
As we did in Remark \ref{rem:Birkhoff_Iwasawa}, 
it suffices to show that 
$Q^{-1}_x A_t \ov{Q}_x$ admits the Birkhoff factorization. 
We have 
\begin{align*}
Q_x^{-1} A_t \ov{Q}_x = B_t (B_t^{-1} Q_x^{-1} B_t) 
(C_t\ov{Q}_x C_t^{-1}) 
C_t   
\end{align*} 
and for $0<\epsilon<\epsilon_0$, 
\[
B_t^{-1} Q^{-1}_x B_t = \unit +  O(e^{-\epsilon t}), \quad 
C_t \ov{Q}_x C_t^{-1} = \unit+ O(e^{-\epsilon t}), \quad 
\text{ as } t=\Re(x) \to \infty.    
\]
Here we used that $B_t$ and $C_t$ 
have at most polynomial growth in $t$. 
By the continuity of Birkhoff factorization, 
$(B_t^{-1} Q^{-1}_x B_t) (C_t\ov{Q}_x C_t^{-1}) = 
\unit +O(e^{-\epsilon t})$ 
admits the Birkhoff factorization of the form: 
\begin{equation}
\label{eq:Birkhoff_withestimate}
(B_t^{-1} Q^{-1}_x B_t) (C_t\ov{Q}_x C_t^{-1}) = 
\tilde{B}_x(z) \tilde{C}_x(z), 
\quad \tilde{B}_x = \unit + O(e^{-\epsilon t}), \quad 
\tilde{C}_x = \unit+ O(e^{-\epsilon t}) 
\end{equation} 
for $t=\Re(x)\gg 0$, where 
$\tilde{B}_x\colon \D_0\to GL_N(\C)$, $\tilde{B}_x(0)=\unit$ 
and $\tilde{C}_x\colon \D_\infty \to GL_N(\C)$. 
The order estimate $O(e^{-\epsilon t})$ holds 
in the $C^0$-norm on the loop space $C^\infty(S^1,\End(\C^N))$. 
See Appendix \ref{subsec:orderestimate_Birkhoff} for 
the proof of the order estimate in (\ref{eq:Birkhoff_withestimate}).   
Therefore $Q_x^{-1} A_t \ov{Q}_x$ also 
has the Birkhoff factorization 
for $t=\Re(x)\gg 0$ and we know that 
\[
 [\Pi_x(\phi_1), \dots,  \Pi_x(\phi_N) ] 
:= [\phi_1,\dots,\phi_N] Q_x(z) B_t(z) \tilde{B}_x(z) 
\] 
form a basis of $\F'_{-x\omega} \cap \kappa^t(\F'_{-x\omega})$, 
\emph{i.e.} 
$e^{-x\omega/z} \Pi_x(\phi_1),\dots, e^{-x\omega/z}\Pi_x(\phi_N)$ 
form a basis of $\F_{-x\omega}\cap \kappa_\cH(\F_{-x\omega})$. 
Using that $\Pi_x(\phi_i) = \varpi_x(\phi_i) + 
O(e^{-\epsilon t})$ 
and Proposition \ref{prop:nilpotentorb_pure}, we have  
\[
(-1)^{(p-q)/2}
(\kappa_{\cH}(e^{-x\omega/z} \Pi_x(\phi)), 
e^{-x\omega/z} \Pi_x(\phi) )_{\cH^\cX} >0, \quad 
\phi\in H^{p,q}(\cX_v)\setminus \{0\} 
\]
for sufficiently big $\Re(x)>0$. 
This completes the proof of Theorem \ref{thm:pure_polarized}.

\subsection{The $\hGamma$-integral (real) structure} 
\label{subsec:Amodel_intstr}
Real or integral structures 
in the quantum cohomology \seminf VHS 
(in the sense of Definition \ref{def:realintstr}) are not unique. 
In this section, we construct an integral structure 
($\hGamma$-integral structure) 
which makes sense for a general symplectic orbifold,  
using $K$-theory. 
Since this satisfies the assumption of 
Theorem \ref{thm:pure_polarized}, 
it yields a Cecotti-Vafa structure 
near the large radius limit point. 
We showed in \cite{iritani-Int} that 
the $\hGamma$-integral structure for a weak Fano toric orbifold 
coincides with the integral structure on the singularity 
mirror (Landau-Ginzburg model) 
\cite{givental-ICM, givental-mirrorthm-toric}. 

Let $K(\cX)$ denote the Grothendieck group 
of topological orbifold vector bundles on $\cX$. 
See \emph{e.g.} \cite{adem-ruan, kawasaki-Vind, moerdijk} 
for vector bundles on orbifolds. 
For simplicity, we assume that $\cX$ is isomorphic 
to the quotient $[Y/G]$ as a topological orbifold 
where $Y$ is a manifold and $G$ is a compact Lie group. 
In this case $K(\cX)$ is a finitely generated 
abelian group \cite{adem-ruan}. 
For example, an orbifold without generic stabilizers 
can be presented as a quotient orbifold $Y/G$ 
(see \emph{e.g.} \cite{adem-ruan}). 
For an orbifold vector bundle $\widetilde{V}$ 
on the inertia stack $I\cX$, 
we have an eigenbundle decomposition of $\widetilde{V}|_{\cX_v}$
\[
\widetilde{V}|_{\cX_v} = \bigoplus_{0\le f<1} 
\widetilde{V}_{v,f} 
\] 
with respect to the stabilizer action over $\cX_v$.    
Here, the stabilizer acts on the component 
$\widetilde{V}_{v,f}$ by $\exp(2\pi\iu f) \in \C$. 
Let $\pr \colon I\cX \to \cX$ be the projection. 
For an orbifold vector bundle $V$, the Chern character 
$\tch(V) \in H^*(I\cX)$ is defined by 
\[
\tch(V) := \bigoplus_{v\in \sfT} \sum_{0\le f<1} e^{2\pi\iu f}
\ch((\pr^*V)_{v,f}) 
\]
where $\ch$ is the ordinary Chern character and 
$V$ is an orbifold vector bundle on $\cX$. 
For an orbifold vector bundle $V$ on $\cX$, 
let $\delta_{v,f,i}$, $i=1,\dots,l_{v,f}$ be the Chern roots of 
$(\pr^*V)_{v,f}$. 
The Todd class $\tTd(V)\in H^*(I\cX)$ is defined by  
\[
\tTd(V) = \bigoplus_{v\in \sfT} 
\prod_{0<f<1,1\le i\le l_{v,f}}\frac{1}{1-e^{-2\pi\iu f}e^{-\delta_{v,f,i}}}
\prod_{f=0,1\le i\le l_{v,0}} \frac{\delta_{v,0,i}}{1-e^{-\delta_{v,0,i}}}
\]
We put $\tTd_\cX := \tTd(T\cX)$. 
For a holomorphic orbifold vector bundle $V$, 
the holomorphic Euler characteristic 
$\chi(V):=\sum_{i=0}^{\dim \cX} (-1)^i \dim H^i(\cX,V) $ 
is given by the Kawasaki-Riemann-Roch formula \cite{kawasaki-rr,toen}:  
\begin{equation}
\label{eq:orbifoldRR}
\chi(V) = \int_{I\cX} \tch(V)\cup \tTd_\cX.    
\end{equation} 
For a not necessarily holomorphic orbifold vector bundle, 
we can use the right-hand side of (\ref{eq:orbifoldRR}) 
as the definition of $\chi(V)$. It follows from 
Kawasaki's $V$-index theorem \cite{kawasaki-Vind} 
that $\chi(V)$ is always an integer 
(see \cite[Remark 2.8]{iritani-Int}). 
We define a multiplicative characteristic class 
$\hGamma \colon K(\cX) \to H^*(I\cX)$ 
called the \emph{$\hGamma$-class} \cite{iritani-Int,KKP} by  
\[
\hGamma(V) := \bigoplus_{v\in \sfT} 
\prod_{0\le f<1} \prod_{i=1}^{l_{v,f}} 
\Gamma(1- f + \delta_{v,f,i}),  
\]
where $\delta_{v,f,i}$ is the same as above.  
The Gamma function on the right-hand side 
should be expanded in series at $1-f>0$. 
This class can be regarded as a ``half" of the 
Todd class. When $\cX$ is a manifold $X$, 
by using $\Gamma(1-z) \Gamma(1+z) 
=  e^{-\pi \iu z} \pi z/(1- e^{-2\pi\iu z})$, 
we have 
\[
e^{\pi\iu c_1(X)}  \cup 
\hGamma(V) \cup (-1)^{\deg/2} \hGamma(V)  
= (2\pi\iu)^{\deg/2} \Td(V). 
\]

\begin{definition-proposition}[{\cite[Proposition 2.10]{iritani-Int}}]
\label{def-prop:A-model_int} 
Put $\hGamma_\cX := \hGamma(T\cX)$. 
Define an integral structure 
$\cV^\cX_\Z\subset \cV^\cX = H_{\rm orb}^*(\cX)$ 
to be the image of the map  
\begin{equation}
\label{eq:stdintstr_A}
\Psi \colon K(\cX) \longrightarrow \cV^\cX, \quad 
[V] \longmapsto \frac{1}{(2\pi)^{n/2}} 
\hGamma_\cX \cup (2\pi\iu)^{\deg/2} \inv^* (\tch(V)),   
\end{equation} 
where $\deg\colon H^*(I\cX)\to H^*(I\cX)$ is a grading operator 
on $H^*(I\cX)$ defined by $\deg = 2k$ on $H^{2k}(I\cX)$ 
and $\cup$ is the cup product in $H^*(I\cX)$. 
Then 
\begin{itemize}
\item[(i)] 
$\cV^\cX_\Z$ is a lattice in $\cV^\cX$ 
such that $\cV^\cX \cong \cV^\cX_\Z \otimes_\Z \C$. 
\item[(ii)]
The Galois action $G^\cV(\xi)$ on $\cV^\cX$ in (\ref{eq:Galois_V}) 
corresponds to tensoring by the line bundle $\otimes L_\xi^\vee$ 
in $K(\cX)$, i.e. 
$\Psi([V\otimes L_\xi^\vee]) = G^\cV(\xi)(\Psi([V]))$.  
\item[(iii)]  
The pairing $(\cdot,\cdot)_{\cV^\cX}$ on $\cV^\cX$ 
in (\ref{eq:pairing_Amodel_V}) 
corresponds to the Mukai pairing on $K(\cX)$ defined by 
$([V_1],[V_2])_{K(\cX)} := \chi(V_2^\vee \otimes V_1)$, i.e. 
$(\Psi([V_1]),\Psi([V_2]))_{\cV^\cX} = ([V_1],[V_2])_{K(\cX)}$.  
In particular, 
the pairing $(\cdot,\cdot)_{\cV^\cX}$ restricted 
on $\cV_\Z^\cX$ takes values in $\Z$. 
\end{itemize} 
Therefore $\cV^\cX_\Z$ and $\cV^\cX_\R := \cV^\cX_\Z\otimes_\Z \R$ 
satisfy the conditions in Proposition \ref{prop:char_A_real_int_str} 
except for the unimodularity of the pairing on $\cV^\cX_\Z$.  
We call $\cV^\cX_\Z$ and $\cV^\cX_\R$ 
the \emph{$\hGamma$-integral structure} and 
the \emph{$\hGamma$-real structure} respectively. 
The real involution $\kappa_\cV$ on $\cV^\cX$ 
for the $\hGamma$-real structure is given by 
\[
\kappa_{\cV} (\alpha) = (-1)^k \prod_{0\le f <1} 
\prod_{i=1}^{l_{\inv(v),f}} \frac{\Gamma(1-f+\delta_{\inv(v),f,i})}
{\Gamma(1-\ov{f}-\delta_{\inv(v),f,i})} \inv^* \ov{\alpha}, \quad 
\alpha \in H^{2k}(\cX_v) \subset \cV^\cX, 
\]
where $\delta_{\inv(v),f,i}$, $i=1,\dots,l_{\inv(v),f}$ are 
the Chern roots of $(\pr^*T\cX)_{\inv(v),f}$ and 
\[
\ov{f} := 
\begin{cases} 
1-f  & \text{if  $0<f<1$} \\
0    & \text{if  $f=0$}. 
\end{cases}
\]
Therefore, this $\kappa_\cV$ satisfies 
(\ref{eq:kappa_induces_Inv}) and (\ref{eq:leadingterm_kappaV}).  
In particular, the conclusions of Theorem \ref{thm:pure_polarized}
hold for the $\hGamma$-real structure 
on the quantum cohomology \seminf VHS. 
\end{definition-proposition} 

The unimodularity of the pairing on $\cV_\Z^\cX$ 
(or on the integral local system $R_\Z$) holds if 
the map 
\[
K(\cX) \to \Hom(K(\cX),\Z) \quad \alpha \mapsto \chi(\alpha \otimes \cdot) 
\] 
is surjective. 
This holds true when $\cX$ is a manifold $X$. 
The author does not know if this holds in general. 

\begin{remark}
Let $\cX=X$ be a manifold. 
A possible origin of the $\hGamma$-class 
might be the Floer theory on the free loop space 
$LX= C^\infty(S^1,X)$. Let $S^1$ act on $LX$ 
by loop rotation. 
Givental's heuristic interpretation \cite{givental-ICM} 
of the quantum $D$-module as the $S^1$-equivariant 
Floer theory suggests that the set $X$ of constant 
loops in $LX$ contributes to the Floer theory by 
the (infinite) localization factor: 
\[
\frac{1}{\Euler_{S^1}(N_+)} = 
\frac{1}{\prod_{m>0} \Euler_{S^1}(TX \otimes \varrho^m)} 
\]
where $N_+ \cong \bigoplus_{m>0} TX \otimes \varrho^m$ 
is the \emph{positive normal bundle} of $X$ in $LX$ 
and $\varrho$ is the one-dimensional 
$S^1$-module of weight 1. 
By the $\zeta$-function regularization, this 
factor gives exactly $z^{-\mu} z^{\rho} (2\pi)^{-n/2} \hGamma_X$, 
where $z = c_1^{S^1}(\varrho)$ is a generator of 
$H^2(BS^1) = H_{S^1}({\rm pt})$. 
\end{remark} 

\section{Example: $tt^*$-geometry of $\Proj^1$} 
\label{sec:exampleP1tt*} 
We calculate the Cecotti-Vafa structure on quantum 
cohomology of $\Proj^1$ with respect to 
the $\hGamma$-real structure 
in Definition-Proposition \ref{def-prop:A-model_int}. 
By \cite[Theorem 4.11]{iritani-Int}, 
the $\hGamma$-real structure here matches with 
a natural real structure on the mirror, 
so the $tt^*$-geometry of $\Proj^1$ 
is the same as that of the Landau-Ginzburg model 
(mirror of $\Proj^1$): 
\[
W_q\colon \C^* \to \C, \quad 
W_q = x + \frac{q}{x}, \quad q\in \C^*.  
\]

Let $\omega\in H^2(\Proj^1)$ be the unique 
integral K\"{a}hler class.  
Let $\{t^0,t^1\}$ be the linear co-ordinate system 
on $H^*(\Proj^1)$ 
dual to the basis $\{\unit, \omega\}$. 
Put $\tau = t^0 \unit + t^1 \omega$. 
The quantum product $\circ_\tau$ is given by 
\[
(\unit \circ_\tau) = 
\begin{bmatrix}
1 & 0 \\
0 & 1 
\end{bmatrix} 
, \quad 
(\omega \circ_\tau) =  
\begin{bmatrix}
0 & e^{t^1} \\ 
1 & 0 
\end{bmatrix}, 
\]
where we identify $\unit,\omega$ with  
column vectors $[1,0]^{\rm T}$, $[0,1]^{\rm T}$ and
the matrices act on vectors by the 
left multiplication. 
The exponential $e^{t^1}$ corresponds to 
$q$ in the Landau-Ginzburg model via the mirror map, 
so we set $q = e^{t^1}$.  
Hereafter, we restrict $\tau$ to lie on $H^2(\Proj^1)$ 
but we will not lose any information by this 
(see Remark \ref{rem:P1tt*} below). 
Recall that the Hodge structure $\F_\tau$ 
associated with the quantum cohomology 
of $\Proj^1$ is given by the image of 
$\J_\tau \colon H^*(\Proj^1) \otimes \cO(\C) 
\to \cH^{\Proj^1} = H^*(\Proj^1) 
\otimes \cO(\C^*)$ in (\ref{eq:fundamentalsol_A}). 
The $J$-function $J(q,z) = \J_\tau\unit$ is given by 
\cite{givental-mirrorthm-toric}:  
\[
J(q,z) :=  e^{t^1 \omega/z} 
\sum_{k=0}^\infty 
\frac{q^k \unit}{(\omega + z)^2 \cdots (\omega+kz)^2}
= e^{t^1\omega/z} (J_0(q,z) \unit + J_1(q,z) \frac{\omega}{z}), 
\]
and the map $\J_\tau$ is given by 
\begin{align*}
&\J_\tau = 
\begin{bmatrix}
\vert & \vert \\ 
J(e^{t^1},z) & z \partial_1 J(e^{t^1},z) \\ 
\vert & \vert  
\end{bmatrix} 
= e^{t^1 \omega/z} \circ Q, \quad 
Q:= 
\begin{bmatrix}
J_0 & z\partial_1 J_0  \\
J_1/z &  J_0 + \partial_1 J_1 
\end{bmatrix} 
\end{align*}  
where $\partial_1 = (\partial/\partial t^1)$. 
By Definition-Proposition \ref{def-prop:A-model_int}, 
an integral basis of $\cV^{\Proj^1}=H^*(\Proj^1)$ is given by 
\[
\Psi(\cO_{\Proj^1})  = 
\frac{1}{\sqrt{2\pi}}(\unit -2\gamma \omega),  \quad 
\Psi(\cO_{\pt})  = \sqrt{2\pi} \iu \omega,   
\] 
where $\gamma$ is the Euler constant. 
Hence the real involutions on $\cV^{\Proj^1}$ and 
$\cH^{\Proj^1}$ are given respectively by 
(see (\ref{eq:relation_kappa_H_V})):  
\[
\kappa_\cV =
\begin{bmatrix}
1 & 0 \\ 
-4\gamma & -1 
\end{bmatrix}\circ \ov{\phantom{A}}, \quad 
\kappa_\cH = 
\begin{bmatrix}
z & 0 \\ 
-4\gamma & -z^{-1} 
\end{bmatrix}\circ \ov{\phantom{A}}.    
\]
where $\ov{\phantom{A}}$ is the usual complex conjugation 
(when $z$ lies in $S^1 =\{|z|=1\}$). 

To obtain the Cecotti-Vafa structure, we need to find 
a basis of $\F_\tau \cap \kappa_\cH(\F_\tau)$. 
The procedure below follows the proof of 
Theorem \ref{thm:pure_polarized} 
in Section \ref{subsec:pure_polarized}. 
Put $\F_\tau ' := e^{-t^1\omega/z}\F_\tau$ 
and $\kappa^\tau_\cH := e^{-(t^1+\ov{t^1})\omega/z}\kappa_\cH$. 
By 
\[ 
\F_\tau \cap \kappa_\cH(\F_\tau) 
= e^{t^1\omega/z} (\F_\tau' \cap \kappa_\cH^\tau (\F'_\tau)), 
\] 
it suffices to calculate a basis of 
$\F_\tau' \cap \kappa_{\cH}^\tau (\F'_\tau)$. 
First we approximate $\F_\tau'$ by 
$\Flim := H^*(\Proj^1)\otimes \cO(\C)$ 
and solve for a basis of $\Flim \cap \kappa^\tau_{\cH}(\Flim)$. 
By elementary linear algebra, we find 
the following Birkhoff factorization of 
$[\kappa^\tau_\cH(\unit), \kappa^\tau_\cH(\omega)]$: 
\[
[\kappa^\tau_\cH(\unit),\kappa^\tau_\cH(\omega)] = BC, \quad 
B:= 
\begin{bmatrix}
1 & z/a_\tau \\ 
0 & 1 
\end{bmatrix}, \quad 
C:= 
\begin{bmatrix}
0 & 1/a_\tau \\
a_\tau & -1/z 
\end{bmatrix}, 
\]
where $a_\tau := -t^1-\ov{t^1} - 4\gamma$. 
Then the column vectors of $B$ give a basis of 
$\Flim \cap \kappa^\tau_{\cH}(\Flim)$ (\emph{cf.} (\ref{eq:A=BC})). 
Note that the column vectors of $Q$ above 
form a basis of $\F'_\tau$. 
Thus the Birkhoff factorization of $Q^{-1}\kappa_\cH^\tau(Q)$ 
calculates a basis of $\F'_\tau \cap \kappa_\cH^\tau(\F'_\tau)$. 
Define a matrix $S$ by 
\[
\kappa_\cH^\tau (Q) = Q B S C. 
\]
Using the fact that $Q^{-1}$ is the adjoint of $Q(-z)$ 
(by Proposition \ref{prop:fundamentalsol_A}), 
we have 
\[
S = 
\begin{bmatrix}
\substack{
2\Re(J_0\ov{J_1})a_\tau^{-1} 
+ |J_0|^2 + 2\Re(\partial_1J_0\ov{J_1}
+J_0\ov{\partial_1J_1}) \\
+ 2  \Re(\partial_1J_0\ov{\partial_1 J_1})a_\tau 
-  |\partial_1J_0|^2 a_\tau^2 \\ 
\phantom{A}
}
&  
\substack{
( 2\Re(J_0\ov{J_1}) a_\tau^{-2} \\    
+ (\partial_1 J_0\ov{J_1}   
+ \ov{J_0}\partial_1 J_1) a_\tau^{-1}   
- \partial_1 J_0\ov{J_0}) z    \\ 
}
\\ 
\substack{
( -2\Re(J_0\ov{J_1}) 
- (\ov{\partial_1 J_0}J_1 + J_0
   \ov{\partial_1 J_1}) a_\tau \\
 + J_0\ov{\partial_1 J_0} a_\tau^2) z^{-1} 
}     
& 
\substack{
-2\Re(J_1\ov{J_0}) a_\tau^{-1} + |J_0|^2
} 
\end{bmatrix}, 
\]
where we restrict $z$ to lie on $S^1=\{|z|=1\}$. 
Because $S = \unit + O(|q|^{1-\epsilon})$, $0<\epsilon<1$ as 
$|q|\to 0$, 
this admits the Birkhoff factorization 
$S= \tilde{B}\tilde{C}$ for $|q|\ll 1$, where 
$\tilde{B}\colon\D_0 \to GL_2(\C)$, 
$\tilde{C}\colon \D_\infty \to GL_2(\C)$ 
such that $\tilde{B}(0) =\unit$. 
Then the column vectors of $QB\tilde{B}=
\kappa_{\cH}^\tau(Q) C^{-1}\tilde{C}^{-1}$ give a 
basis of $\F_\tau'\cap \kappa_\cH^\tau(\F_\tau')$. 
We can perform the Birkhoff factorization in the following way. 
Note that $S$ is expanded in a power series in $q$ and $\ov{q}$ 
with coefficients in Laurent polynomials in $a_\tau$ and 
$z$: 
\[
S = \sum_{n,m\ge 0} 
S_{n,m} q^n \ov{q}^m, \quad 
S_{n,m} \in 
\End(\C^2)[z,z^{-1},a_\tau,a_\tau^{-1}]. 
\]
We put $\tilde{B} = \sum_{n,m\ge 0}
\tilde{B}_{n,m}q^n\ov{q}^m$, 
$\tilde{C} = \sum_{n,m\ge 0} \tilde{C}_{n,m} q^n \ov{q}^m$. 
Since $S_{0,0}=\tilde{B}_{0,0}=\tilde{C}_{0,0}=\id$, 
we can recursively solve for 
$\tilde{B}_{n,m}$ and $\tilde{C}_{n,m}$ 
by decomposing 
\[
\tilde{B}_{n,m} + \tilde{C}_{n,m} = S_{n,m} - 
\sum_{(i,j)\neq 0, (n-i,m-j)\neq 0} 
\tilde{B}_{i,j} \tilde{C}_{n-i,m-j} 
\] 
into strictly positive power series 
$\tilde{B}_{n,m}$ and 
non-positive power series $\tilde{C}_{n,m}$ in $z$.   
The first six terms of $B\tilde{B}$ are given by 
\begin{align*}
B\tilde{B} &= 
\begin{bmatrix}
1 & \frac{z}{a_\tau} \\ 
0 & 1 
\end{bmatrix} 
+ \ov{q}
\begin{bmatrix} 
(1+a_\tau)z^2 & \frac{z^3}{a_\tau} \\ 
(2+2a_\tau + a_\tau^2) z & \frac{(2+a_\tau)z^2}{a_\tau}  
\end{bmatrix} 
+ q\ov{q}
\begin{bmatrix}
0 & -\frac{(8+8a_\tau+2 a_\tau^2)z}{a_\tau^2} \\
0 & 0 
\end{bmatrix} \\
+ & \ov{q}^2 
\begin{bmatrix}
\frac{(1+2a_\tau)z^4}{4} & \frac{z^5}{4a_\tau} \\
\frac{(3+6a_\tau+2a_\tau^2)z^3}{4} & 
\frac{(3+a_\tau)z^4}{4a_\tau}  
\end{bmatrix} 
+ q \ov{q}^2 
\begin{bmatrix}
\frac{(33+34a_\tau+18a_\tau^2+4a_\tau^3)z^2}{4} & 
-\frac{(32+31a_\tau+12a_\tau^2 +2a_\tau^3)z^3}{4a_\tau^2} \\
\frac{(25+50 a_\tau +34a_\tau^2 +12 a_\tau^3 +2 a_\tau^4)z}{2} &
-\frac{(64+78a_\tau+ 45 a_\tau^2+14 a_\tau^3 +2 a_\tau^4)z^2}
{4a_\tau^2} 
\end{bmatrix}
\\ 
+ &\ov{q}^3 
\begin{bmatrix}
\frac{(1+3a_\tau)z^6}{36} & \frac{z^7}{36a_\tau} \\
\frac{(11+33a_\tau+9a_\tau^2)z^5}{108} & 
\frac{(11+3a_\tau)z^6}{108a_\tau} 
\end{bmatrix}  
+ O((\log|q|)^5 |q|^4) 
\end{align*} 
Let $\Phi_\tau$ denote the inverse 
to the natural projection 
$\F_\tau \cap \kappa_{\cH}(\F_\tau) \to \F_\tau/z\F_\tau 
= H^*(\Proj^1)$. 
Because $B\tilde{B}=\unit+O(z)$, 
we have 
$[\Phi_\tau(\unit),\Phi_\tau(\omega)] 
= e^{t^1\omega/z} QB\tilde{B}$: 
\[
\begin{CD}
\Phi_\tau \colon H^*(\Proj^1) = 
\F_\tau'/z\F_\tau' @>{QB\tilde{B}}>>
\F'_\tau \cap \kappa_{\cH}^\tau(\F_\tau') 
@>{e^{t^1\omega/z}}>> 
\F_\tau \cap \kappa_{\cH}(\F_\tau).    
\end{CD} 
\] 
The Cecotti-Vafa structure for $\Proj^1$ 
is defined on the trivial vector bundle 
$K := H^*(\Proj^1)\times H^*(\Proj^1) \to H^*(\Proj^1)$. 
Recall that the Hermitian metric $h$ 
on $K_\tau$ is the pull-back of the Hermitian metric 
$(\alpha, \beta) \mapsto 
(\kappa_{\cH}(\alpha),\beta)_{\cH}$ 
on $\F_\tau \cap \kappa_\cH(\F_\tau)$ 
through $\Phi_\tau \colon K_\tau 
\cong \F_\tau \cap \kappa_\cH(\F_\tau)$. 
The Hermitian metric $h$ is of the form: 
\[
h = 
\begin{bmatrix}
h_{\ov{0}0} & 0 \\
0 & h_{\ov{0}0}^{-1}  
\end{bmatrix}, \quad 
h_{\ov{0}0}:= \int_{\Proj^1} 
\kappa_{\cH}(\Phi_\tau(\unit))\Big|_{z\mapsto -z} 
\cup \Phi_\tau(\unit). 
\] 
The first seven terms of the expansion of $h_{\ov{0}0}$ 
are (with $a_\tau = -t^1-\ov{t^1}-4\gamma$, $q=e^{t^1}$) 
\scriptsize 
\begin{align*} 
h_{\ov{0}0} = & a_\tau+ 
|q|^{2} 
\left( {a_\tau^3}+4{a_\tau^2}+8 a_\tau+8 \right) + 
|q| ^{4} 
\left( {a_\tau^5}+ 8{a_\tau^4}+{\frac{121}{4}}{a_\tau^3}+
{\frac {129}{2}}{a_\tau^2}+{\frac 
{145}{2}} a_\tau+{\frac {145}{4}} \right) \\ 
&+ |q|^{6} 
\left( a_\tau^{7}+12 a_\tau^{6}+{\frac {275}{4}}a_\tau^{5}+{
\frac {477}{2}}a_\tau^{4}+{\frac {9539}{18}}a_\tau^{3}  
+ {\frac {81001}{108}} 
a_\tau^{2}+{\frac {50342}{81}}a_\tau 
+{\frac {55526}{243}} \right) \\ 
& +|q|^{8} 
\left( a_\tau^{9}+16a_\tau^{8}+{
\frac {493}{4}}a_\tau^{7}+{\frac {1185}{2}}a_\tau^{6}+{\frac {31001}{16}
}a_\tau^{5}+{\frac {79939}{18}}a_\tau^{4}+{\frac {49077907}{6912}}a_\tau^{3}
\right. \\ 
& \qquad \quad \left. 
+{\frac {52563371}{6912}}a_\tau^{2}+{\frac {614694323}{124416}}a_\tau+{
\frac {736622003}{497664}} \right)  \\
& + |q|^{10}
\left( a_\tau^{11}+20a_\tau^{10}+{\frac {775}{4}}a_\tau^{9}+
{\frac {2381}{2}}a_\tau^{8}+{\frac {368599}{72}}a_\tau^{7}+{\frac {
1738481}{108}}a_\tau^{6}+{\frac {780126811}{20736}}a_\tau^{5} 
\right. \\ 
& \qquad \quad +\left. {\frac {
4053627445}{62208}}a_\tau^{4}+{\frac {254355946241}{3110400}}a_\tau^{3}+
{\frac {1465574917127}{20736000}}a_\tau^{2}+{\frac {163291639271}{
4320000}}a_\tau +{\frac {1840366543439}{194400000}} \right) \\
&+|q| ^{12} 
\left( a_\tau^{13}+24a_\tau^{12}+{\frac {
1121}{4}}a_\tau^{11}+{\frac {4193}{2}}a_\tau^{10}+{\frac {1606399}{144}}
a_\tau^{9}+{\frac {2398517}{54}}a_\tau^{8}+{\frac {2814667745}{20736}}
a_\tau^{7}+{\frac {20004983519}{62208}}a_\tau^{6} \right. \\ 
& \qquad \quad +{\frac {407437321759}{
691200}}a_\tau^{5}+{\frac {51278023471273}{62208000}}a_\tau^{4}+{\frac {
796478452045403}{933120000}}a_\tau^{3}+{\frac {11553263487112967}{
18662400000}}a_\tau^{2} \\ 
& \qquad \quad \left. +{\frac {11823418405646927}{41990400000}}a_\tau
+{\frac{15268380040196927}{251942400000}} \right)
+ \cdots. 
\end{align*} 
\normalsize 
The other data 
$(\kappa,g, C,\tilde{C},D,\cU,\ov{\cU},\cQ)$ 
of the Cecotti-Vafa structure 
are given in terms of $h_{\ov{0}0}$. 
In fact, we have $C_0= \tC_{\ov{0}} =\id$, 
$D_0=\partial/\partial t^0$, 
$D_{\ov{0}} = \partial/\partial \ov{t^0}$ and 
\begin{gather*}
g = 
\begin{bmatrix}
0 & 1 \\
1 & 0
\end{bmatrix},  
\quad 
\kappa = 
\begin{bmatrix}
0 & h_{\ov{0}0}^{-1}  \\
h_{\ov{0}0} & 0
\end{bmatrix} \circ \ov{\phantom{A}},  
\quad 
D_1 = \partial_1 + 
\begin{bmatrix}
\partial_1 \log h_{\ov{0}0} & 0 \\
0 & -\partial_1 \log h_{\ov{0}0} 
\end{bmatrix}, \\  
D_{\ov{1}} = \ov{\partial_1}, \quad  
C_1 = \frac{1}{2} \cU = 
\begin{bmatrix}
0 & e^{t^1} \\
1 & 0  
\end{bmatrix}, \quad 
\tC_{\ov{1}} = 
\frac{1}{2} \ov\cU =
\begin{bmatrix}
0 &  h_{\ov{0}0}^{-2} \\
e^{\ov{t^1}} h_{\ov{0}0}^2 & 0  
\end{bmatrix},  \\ 
\cQ =  \partial_E+ \mu - D_E   
= 
\begin{bmatrix}
-\frac{1}{2} - 2 \partial_1 \log h_{\ov{0}0} & 0 \\
0 & \frac{1}{2} +2 \partial_1 \log h_{\ov{0}0} 
\end{bmatrix},    
\end{gather*} 
where $\partial, \ov{\partial}$ are 
the connections given by the given trivialization 
of $K$. 

\begin{remark}
\label{rem:P1tt*} 
(i) Takahashi \cite{takahashi-tt*} classified 
the $tt^*$-geometry of rank 2. 
The $\Proj^1$ case is included in the consideration 
in Section 5 \emph{ibid.},  
but this does not seem to appear in Theorem 5.1 \emph{ibid}.  
It is shown in Lemma 2.1 \emph{ibid}.\   
that the Hermitian metric $h$ is 
represented by a diagonal matrix with determinant $1$. 
 
(ii) From the theory 
of (trTERP)$+$(trTLEP) structure on the tangent bundle, 
it follows that $h,C,\tC,\cU,\ov{\cU},\cQ$ are invariant 
under the flow of the unit vector field 
$(\partial/\partial t^0), (\partial/\partial \ov{t^0})$. 
Therefore, the calculation here 
determines the Cecotti-Vafa structure on the big 
quantum cohomology. Moreover we have 
$D_E + \cQ = \partial_E + \mu$ and $\Lie_{E-\ov{E}} h =0$. 
In the case of $\Proj^1$, this means that $h_{\ov{0}0}$  
depends only on $|q|$. See \cite{hertling-tt*}.  

(iii) We can show that our procedure for 
the Birkhoff factorization gives convergent 
series for sufficiently small values of $|q|$. 
In particular, the expansion for 
$h_{\ov{0}0}$ converges for small $|q|$. 

(iv) Since the preprint version \cite{iritani-realint-preprint} 
of this paper was written, Dorfmeister-Guest-Rossman \cite{DGR} 
found that the $tt^*$-geometry of $\Proj^1$ 
gives a new example of a CMC surface 
in Minkowski space $\R^{2,1}$. 
\end{remark} 

We explain a different way to calculate 
the Hermitian metric $h_{\ov{0}0}$ 
due to Cecotti-Vafa \cite{cecotti-vafa-exactsigma}. 
The $tt^*$-equation 
$[D_1,D_{\ov{1}}] + [C_1,\tC_{\ov{1}}] =0$ 
(see Proposition \ref{prop:CV-str}) 
gives the following differential equation 
for $h_{\ov{0}0}$:  
\begin{equation}
\label{eq:diffeq_h00}
\partial_1 \ov{\partial_{1}} 
\log h_{\ov{0}0} = - 
h_{\ov{0}0}^{-2} + |q|^2 h_{\ov{0}0}^2. 
\end{equation} 
Cecotti-Vafa \cite{cecotti-vafa-exactsigma} 
identified $h_{\ov{0}0}$ with a unique solution 
to (\ref{eq:diffeq_h00}) expanded in the form 
\[
h_{\ov{0}0} = \sum_{n=0}^\infty F_n|q|^{2n}, \quad 
F_0 = a_\tau, \quad F_n\in \C[a_\tau,a_\tau^{-1}], 
\quad a_\tau = -2 \log |q| -4\gamma. 
\]  
The equation (\ref{eq:diffeq_h00}) gives an infinite 
set of recursive differential equations for $F_n$. 
It is easy to check that the differential equations 
determine the Laurent polynomial $F_n$ \emph{uniquely}. 
Moreover it turns out that 
$F_n\in \Q[a_\tau]$ and $\deg F_n = 2n+1$. 
The \emph{existence} of such a solution 
seems to be non-trivial, but the Birkhoff factorization 
certainly gives such $h_{\ov{0}0}$. 
By physical arguments, Cecotti-Vafa 
\cite{cecotti-vafa-top-antitop, cecotti-vafa-classification, 
cecotti-vafa-exactsigma} 
showed that $h_{\ov{0}0}$ should be positive and smooth 
on the positive real axis $0< |q| <\infty$\footnote{
For this, the constant $\gamma$ in 
$a_\tau$ must be the very Euler constant.}. 
Since the Landau-Ginzburg mirror of $\Proj^1$ 
is given by a cohomologically tame function, 
this follows from the Sabbah's result \cite{sabbah} 
in singularity theory (see Remark \ref{rem:sabbah}). 
Therefore, the Cecotti-Vafa structure for $\Proj^1$ 
is well-defined and positive definite 
on the whole $H^*(\Proj^1)$. 

Cecotti-Vafa \cite{cecotti-vafa-top-antitop} 
also found that the differential equation 
(\ref{eq:diffeq_h00}) is equivalent to the 
Painlev\'{e} III equation: 
\[
\frac{d^2u}{dz^2} + \frac{1}{z} \frac{du}{dz} = 4 
\sinh(u), \quad h_{\ov{0}0} = e^{u/2} |e^{-t^1/2}|, 
\quad z = 4 |e^{t^1/2}|.  
\] 
It seems that the solution corresponding to 
our $h_{\ov{0}0}$ has been obtained in the study of 
Painlev\'{e} III equation 
\cite{its-novokshenov, mccoy-tracy-wu-PIII} 
(in fact, the first few terms of the expansion 
of their solutions match with ours).  
If this is the case, $h_{\ov{0}0}$ should have 
the asymptotics \cite{its-novokshenov, mccoy-tracy-wu-PIII} 
(also appearing in \cite{cecotti-vafa-exactsigma}):  
\[ 
h_{\ov{0}0} \sim \frac{1}{\sqrt{|q|}} 
\left(1-\frac{1}{2\sqrt{\pi}|q|^{1/4}}e^{-8|q|^{1/2}}\right) 
\] 
as $|q|\to \infty$. 
With respect to the metric $h_{\ov{1}1}=h_{\ov{0}0}^{-1}$ 
on the K\"{a}hler moduli space $H^2(\Proj^1)/2\pi\iu 
H^2(\Proj^1,\Z)$, a neighborhood of the 
large radius limit point $q=0$ has negative curvature, 
but does not have finite volume. 
The curvature 
$-\frac{2}{h_{\ov{0}0}}(1-|q|^2 h_{\ov{0}0}^4)$  
goes to zero as $|q|\to 0$ and $|q|\to \infty$ 
and the total curvature is $-\pi/4$. 
Much more examples including $\Proj^n$, $\Proj^1/\Z_n$ 
are calculated in physics literature. 
We refer the reader to    
\cite{cecotti-vafa-top-antitop,cecotti-vafa-exactsigma, 
cecotti-vafa-massiveorb}. 

\section{Appendix} 
\subsection{Proof of (\ref{eq:Birkhoff_withestimate})} 
\label{subsec:orderestimate_Birkhoff} 
Birkhoff's theorem implies that 
there exists an open dense neighborhood of $\unit$ 
in the loop group $LGL_N(\C)$ 
which is diffeomorphic to the product of subgroups
$L^+_1GL_N(\C)\times L^-GL_N(\C)$ \cite{pressley-segal}.  
We use the inverse function theorem for Hilbert manifolds 
to explain the order estimate in 
(\ref{eq:Birkhoff_withestimate}). 
Consider the space $LGL_N(\C)^{1,2}$ of Sobolev loops 
which consists of maps $\lambda\colon S^1 \to GL_N(\C)$ such 
that $\lambda$ and its weak derivative $\lambda'$ are 
square integrable. 
Note that this is a subgroup of 
the group of continuous loops by 
Sobolev embedding theorem 
$W^{1,2}(S^1) \subset C^0(S^1)$ and 
the multiplication theorem 
$W^{1,2}(S^1)\times W^{1,2}(S^1) \to W^{1,2}(S^1)$. 
$LGL_N(\C)^{1,2}$ is a Hilbert manifold 
modeled on the Hilbert space $W^{1,2}(S^1,\gl_N(\C))$. 
A co-ordinate chart of a neighborhood of $\unit$ is 
given by the exponential map $A(z) \mapsto e^{A(z)}$. 
Let $L^+_1GL_N(\C)^{1,2}$ be the subgroup of 
$LGL_N(\C)^{1,2}$ consisting of the boundary values 
of holomorphic maps $\lambda_+ \colon 
\{|z|<1\} \to GL_N(\C)$ satisfying $\lambda_+(0)=\unit$. 
Let $L^-GL_N(\C)^{1,2}$ be the subgroup of 
$LGL_N(\C)^{1,2}$ consisting of the boundary values 
of holomorphic maps $\lambda_- \colon 
\{|z|>1\}\cup\{\infty\} \to GL_N(\C)$. 
Notice that $W^{1,2} := W^{1,2}(S^1,\gl_N(\C))$ has 
the direct sum decomposition: 
\begin{equation}
\label{eq:Liealgdecomp_pos_neg}
W^{1,2} = W^{1,2}_+ \oplus W^{1,2}_-, \quad  
\end{equation} 
where $W^{1,2}_+$ ($W^{1,2}_-$) 
is the closed subspace of 
$W^{1,2}(S^1,\gl_N(\C))$ consisting of 
strictly positive Fourier series $\sum_{n>0} a_n z^n$ 
(non-positive Fourier series $\sum_{n\le 0} a_nz^n$ resp.) 
with $a_n\in \gl_N(\C)$. 
The subgroups $L^+_1GL_N(\C)^{1,2}$ and $L^-GL_N(\C)^{1,2}$ are 
modeled on the Hilbert spaces $W^{1,2}_+$ and 
$W^{1,2}_-$ respectively. 
Consider the multiplication map 
$L^+_1GL_N(\C)^{1,2} \times L^-GL_N(\C)^{1,2}
\to LGL_N(\C)^{1,2}$. 
The differential of this map at the identity is given by 
the sum $W^{1,2}_+ \times W^{1,2}_- \to W^{1,2}$ 
and is clearly an isomorphism. 
By the inverse function theorem for Hilbert manifolds, 
there exists a differentiable inverse map 
on a neighborhood of $\unit$.  
In the case at hand, we have  
$\|(B_t^{-1}Q_t B_t) (C_t \ov{Q}_t C_t^{-1})-\unit\|_{W^{1,2}}
= O(e^{-\epsilon t})$ as $t\to \infty$.  
Therefore, this admits the Birkhoff factorization 
(\ref{eq:Birkhoff_withestimate}) for $t\gg 0$ with 
$\|\tilde{B}_t-\unit\|_{W^{1,2}} = O(e^{-\epsilon t})$ and 
$\|\tilde{C}_t - \unit\|_{W^{1,2}} = O(e^{-\epsilon t})$.  
By Sobolev embedding, the order estimates hold also for 
the $C^0$-norm. 
(The method here does not work directly 
for the Banach manifold of continuous loops, 
since the decomposition (\ref{eq:Liealgdecomp_pos_neg}) 
is not true in this case.)

\bibliographystyle{amsplain}

\begin{thebibliography}{alpha}

\bibitem{AGV} 
Abramovich, Dan; Graber, Tom; Vistoli, Angelo   
\emph{Gromov-Witten theory of Deligne-Mumford stacks.} 
preprint, arXiv:math.AG/0603151. 

\bibitem{adem-ruan} 
Adem, Alejandro; Ruan, Yongbin 
\emph{Twisted orbifold $K$-theory.}
Comm.\ Math.\ Phys.\ 237 (2003), no. 3, pp.533--556. 




\bibitem{barannikov-qpI}
Barannikov, Serguei 
\emph{Quantum periods. I. 
        Semi-infinite variations of Hodge structures.}
Internat.\ Math.\ Res.\ Notices 2001, no. 23, pp.1243--1264.

\bibitem{barannikov-proj} 
Barannikov, Serguei 
\emph{Semi-infinite Hodge structures and mirror symmetry 
      for projective spaces.} 
preprint, arXiv:math.AG/0010157. 














\bibitem{cecotti-vafa-top-antitop}
Cecotti, Sergio; 
Vafa, Cumrun 
\emph{Topological--anti-topological fusion.} 
Nuclear Phys.\ B 367 (1991), no. 2, pp.359--461.


\bibitem{cecotti-vafa-exactsigma} 
Cecotti, Sergio; 
Vafa, Cumrun 
\emph{Exact results for supersymmetric $\sigma$ models.} 
Phys.\ Rev.\ Lett.\ 68 (1992), no. 7, pp.903--906.


\bibitem{cecotti-vafa-massiveorb} 
Cecotti, Sergio; 
Vafa, Cumrun 
\emph{Massive orbifolds.} 
Modern Phys.\ Lett.\ A 7 (1992), no. 19, 
pp.1715--1723.


\bibitem{cecotti-vafa-classification} 
Cecotti, Sergio; 
Vafa, Cumrun 
\emph{On classification of $N=2$ supersymmetric theories.}  
Comm.\ Math.\ Phys. 158 (1993), no. 3, 
pp.569--644. 



\bibitem{chen-ruan:new_coh_orb}
Chen, Weimin; Ruan, Yongbin  
\emph{A new cohomology theory of orbifold.}
Comm.\ Math.\ Phys.\ B 359 (1991) no.1, pp.1--31. 


\bibitem{chen-ruan:GW}  
Chen, Weimin; Ruan, Yongbin  
\emph{Orbifold Gromov-Witten theory.}
Orbifolds in mathematics and physics (Madison, WI, 2001), 
Contemp.\ Math., vol. 310, 
Amer.\ Math.\ Soc., Province, RI, 2002, pp.25-85.  







\bibitem{coates-givental} 
Coates, Tom; 
Givental, Alexander B.  
\emph{Quantum Riemann-Roch, Lefschetz and Serre.}
Ann.\ of Math. (2) 165 (2007), no. 1, pp.15--53. 


\bibitem{CIT:I} 
Coates, Tom; 
Iritani, Hiroshi; 
Tseng, Hsian-Hua  
\emph{Wall-crossings in toric Gromov-Witten theory I: crepant examples.}
preprint, arXiv:math.AG/0611550, to appear in Geometry and Topology.  




\bibitem{cox-katz} 
Cox, David A.; 
Katz, Sheldon   
\emph{Mirror symmetry and algebraic geometry.} 
Mathematical Surveys and Monographs, 68. 
American Mathematical Society, Providence, RI, 1999.



\bibitem{DGR} 
Dorfmeister, Josef; 
Guest, Martin; 
Rossman, Wayne 
\emph{The $tt^*$ structure of the quantum cohomology 
of $\C\Proj^1$ from the viewpoint of differential 
geometry.} 
preprint, arXiv:0905.3876.  


\bibitem{dubrovin-2D} 
Dubrovin, Boris 
\emph{Geometry of $2D$ topological field theories.} 
Integrable systems and quantum groups 
(Montecatini Terme, 1993), pp.120--348, 
Lecture Notes in Math.\ 1620, Springer, Berlin, 1996.


\bibitem{dubrovin-fusion}
Dubrovin, Boris 
\emph{Geometry and integrability of 
topological-antitopological fusion.} 
Comm.\ Math.\ Phys. 152 (1993), no. 3, pp.539--564.








\bibitem{givental-ICM}
Givental, Alexander B. 
\emph{Homological geometry and mirror symmetry.} 
Proceedings of the International Congress of Mathematicians, 
Vol. 1, 2 (Z\"{u}rich, 1994), pp.472--480, 
Birkh\"{a}user, Basel, 1995. 


\bibitem{givental-mirrorthm-toric} 
Givental, Alexander B.  
\emph{A mirror theorem for toric complete intersections.} 
Topological field theory, primitive forms and related topics 
(Kyoto, 1996), pp.141--175, 
Progr.\ Math., 160, 
Birkh\"{a}user Boston, Boston, MA, 
1998. 


\bibitem{guest_durham} 
Guest, Martin 
Lecture at ``Method of Integrable Systems 
in Geometry", Durham, August, 2006.  

\bibitem{guest-qc_int} 
Guest, Martin 
\emph{From quantum cohomology to integrable systems.} 
Oxford Graduate Texts in Mathematics, 15. 
Oxford University Press, Oxford, 2008. 

\bibitem{hertling-tt*} 
Hertling, Claus 
\emph{$tt^*$-geometry, Frobenius manifolds, their connections, 
and the construction for singularities.}
J.\ Reine Angew.\ Math. 555 (2003), pp.77--161. 


\bibitem{hertling-sevenheck} 
Hertling, Claus; 
Sevenheck, Christian 
\emph{Nilpotent orbits of 
a generalization of Hodge structures.} 
J.\ Reine Angew.\ Math. 609 (2007), pp.23--80. 











\bibitem{iritani-realint-preprint} 
Iritani, Hiroshi 
\emph{Real and integral structure in quantum cohomology I: 
toric orbifolds.} 
preprint, arXiv:0712.2204. 


\bibitem{iritani-Int} 
Iritani, Hiroshi 
\emph{An integral structure of quantum cohomology 
and mirror symmetry for toric orbifolds.} 
preprint, arXiv:0903.1463, to appear in Adv.\ Math. 


\bibitem{its-novokshenov}
Its, Alexander R.; 
Novokshenov, Victor Yu. 
\emph{The isomonodromic deformation method 
in the theory of Painlev\'{e} equations.} 
Lecture Notes in Mathematics, 1191. 
Springer-Verlag, Berlin, 1986.   


\bibitem{KKP} 
Katzarkov, Ludmil; 
Kontsevich, Maxim; 
Pantev, Tony 
\emph{Hodge theoretic aspects of mirror symmetry.} 
From Hodge theory to integrability and TQFT $tt^*$-geometry, 
pp. 87--174, Proc. Sympos. Pure Math., 
78, Amer.\ Math.\ Soc., Providence, RI, 2008, 
available at arXiv:0806.0107. 


\bibitem{kawasaki-rr} 
Kawasaki, Tetsuro  
\emph{The Riemann-Roch theorem for complex $V$-manifolds.} 
Osaka J.\ Math., 16, 1979, pp.151--159.    


\bibitem{kawasaki-Vind}
Kawasaki, Tetsuro  
\emph{The index of elliptic operators over V-manifolds.} 
Nagoya Math.\ J., 84, 1981, pp.135--157




\bibitem{manin} 
Manin, Yuri I.   
\emph{Frobenius manifolds, quantum cohomology and moduli spaces.}
American Mathematical Society Colloquium Publications, 47. 
American Mathematical Society, Providence, RI, 1999.


\bibitem{mccoy-tracy-wu-PIII}  
McCoy, Barry M.; 
Tracy, Craig A.; 
Wu, Tai Tsun 
\emph{Painlev\'{e} functions of the third kind.} 
J.\ Mathematical Phys. 18 (1977), no. 5, pp.1058--1092. 


\bibitem{moerdijk} 
Moerdijk, Ieke    
\emph{Orbifolds as groupoids: an introduction.} 
Orbifolds in mathematics and physics (Madison, WI, 2001), 
pp.205--222, Contemp.\ Math., 310, Amer.\ Math.\ Soc., 
Providence, RI, 2002.




\bibitem{morrison-mathaspects} 
Morrison, David R. 
\emph{Mathematical aspects of mirror symmetry.} 
Complex algebraic geometry (Park City, UT, 1993), 
pp.265--327, 
IAS/Park City Math.\ Ser., 3, 
Amer.\ Math.\ Soc., Providence, RI, 1997. 









\bibitem{pressley-segal} 
Pressley, Andrew; 
Segal, Graeme   
\emph{Loop groups.} 
Oxford Mathematical Monographs, 
The Clarendon Press, Oxford University Press, 
1986. 






\bibitem{sabbah} 
Sabbah, Claude 
\emph{Fourier-Laplace transform of 
a variation of polarized complex Hodge structure.}
J.\ Reine Angew.\ Math. 621 (2008), pp.123--158.


\bibitem{schmid} 
Schmid, Wilfried 
\emph{Variation of Hodge structure: 
the singularities of the period mapping.} 
Invent.\ Math. 22 (1973), pp.211--319. 


\bibitem{simpson-mixedtwistor} 
Simpson, Carlos T. 
\emph{Mixed twistor structures.} 
preprint, arXiv:math.AG/9705006. 


\bibitem{takahashi-tt*} 
Takahashi, Atsushi 
\emph{$tt^*$-geometry of rank 2.} 
Internat.\ Math.\ Res.\ Notices, 2004, no.22.   


\bibitem{toen} 
To\"{e}n, Bertrand  
\emph{Th\'{e}or\`{e}mes de Riemann-Roch 
pour les champs de Deligne-Mumford.}  
$K$-Theory 18 (1999), no. 1, pp.33--76. 


\bibitem{tseng:QRR} 
Tseng, Hsian-Hua 
\emph{Orbifold quantum Riemann-Roch, Lefschetz and Serre.} 
preprint, arXiv:math.AG/0506111. 
\end{thebibliography}

\end{document}